\documentclass[journal,twoside,web]{IEEEtran}
\usepackage[table,pdftex,hyperref,svgnames]{xcolor}  
\usepackage{mathrsfs,xspace,multicol,wrapfig}
\usepackage[utf8]{inputenc}
\usepackage{amsmath}
\usepackage{amssymb}
\usepackage{amsthm}
\usepackage{graphicx}
\usepackage[shortlabels]{enumitem}
\usepackage{here}
\usepackage{multirow}
\usepackage{stmaryrd}

\usepackage{version}
\includeversion{arxiv}\excludeversion{tac}\excludeversion{extensions}

\begin{tac}
\end{tac}
\begin{arxiv}
\end{arxiv}

\newtheorem{theorem}{Theorem}
\newtheorem{proposition}[theorem]{Proposition}

\newtheorem{corollary}[theorem]{Corollary}
\newtheorem{lemma}[theorem]{Lemma}
\newtheorem*{lemma*}{Lemma}
\theoremstyle{definition}
\newtheorem{definition}[theorem]{Definition}
\newtheorem{defn}[theorem]{Definition}
\newtheorem{remark}[theorem]{Remark}
\newtheorem{example}[theorem]{Example}

\newtheorem*{example*}{Example}

\newcommand{\R}{\mathbb{R}}

\newcommand{\sign}{\operatorname{sign}}

\DeclareMathOperator{\diag}{diag}

\newcommand{\mcX}{\mathcal{X}}
\newcommand{\mcU}{\mathcal{U}}

\usepackage{hyperref}%
\hypersetup{colorlinks=true, linkcolor=blue, breaklinks=true, urlcolor=blue}

\newcommand{\until}[1]{\{1,\dots, #1\}}
\newcommand{\subscr}[2]{#1_{\textup{#2}}}

\newcommand{\setdef}[2]{\{#1 \; | \; #2\}}

\newcommand{\bigsetdef}[2]{\big\{#1 \; | \; #2\big\}}
\newcommand{\Bigsetdef}[2]{\Big\{#1 \; | \; #2\Big\}}
\newcommand{\map}[3]{#1: #2 \rightarrow #3}

\newcommand{\real}{\ensuremath{\mathbb{R}}}
\newcommand{\realpositive}{\ensuremath{\mathbb{R}}_{>0}}

\newcommand{\realnonnegative}{\ensuremath{\mathbb{R}}_{\ge 0}}

\newcommand\oprocendsymbol{\hbox{$\triangle$}}
\newcommand\oprocend{\relax\ifmmode\else\unskip\hfill\fi\oprocendsymbol}

\renewcommand{\theenumi}{(\roman{enumi})}

\DeclareSymbolFont{bbold}{U}{bbold}{m}{n}
\DeclareSymbolFontAlphabet{\mathbbold}{bbold}
\newcommand{\vect}[1]{\mathbbold{#1}}
\newcommand{\vectorones}[1][]{\vect{1}_{#1}}
\newcommand{\vectorzeros}[1][]{\vect{0}_{#1}}

\newcommand{\ds}{\displaystyle}

\newcounter{saveenum}

\usepackage[prependcaption,colorinlistoftodos]{todonotes}

\newcommand*{\mydoi}[1]{\href{http://dx.doi.org/#1}{\includegraphics[width=.75em]{doi.png}}}

\newcommand{\norm}[2]{\|#1\|_{#2}}

\newcommand{\Iinfty}{I_{\infty}}

\newcommand{\osL}{\operatorname{osL}}
\newcommand{\id}{\operatorname{Id}}

\newcommand{\jac}[1]{D\mkern-2.5mu{#1}}
\newcommand{\jacave}[1]{\overline{D\mkern-3mu{#1}}}
\newcommand{\WSIP}[2]{\left\llbracket{#1}, {#2}\right\rrbracket}

\newcommand{\GB}{Gr\"onwall\xspace}
\newcommand{\seminorm}[1]{{\left\vert\kern-0.25ex\left\vert\kern-0.25ex\left\vert #1 
		\right\vert\kern-0.25ex\right\vert\kern-0.25ex\right\vert}}
\DeclareMathOperator{\kernel}{Ker}
\newcommand{\kerperp}[1]{\kernel{#1}^\perp}
\newcommand{\semimeasure}[1]{\mu_{\seminorm{\cdot}}\kern-0.5ex\left(#1\right)}

\begin{document}

\title{Non-Euclidean Contraction Theory \\ for Robust Nonlinear Stability\thanks{This material is based upon work supported by the National Science Foundation Graduate Research Fellowship under Grant No.\ 2139319, AFOSR grant FA9550-22-1-0059, and the Defense Threat Reduction Agency
under Contract No. HDTRA1-19-1-0017.}}

\author{Alexander Davydov, \IEEEmembership{Student~Member,~IEEE}, Saber Jafarpour, \IEEEmembership{Member,~IEEE}, and Francesco Bullo, \IEEEmembership{Fellow,~IEEE}
\thanks{Authors are with the Department of Mechanical Engineering and the Center for Control, Dynamical Systems, and Computation, University of California, Santa Barbara, 93106-5070, USA. {({\tt \{davydov, saber, bullo\}@ucsb.edu})}}
}
\maketitle
\thispagestyle{empty}
\pagestyle{empty}

\begin{abstract}
  
  We study necessary and sufficient conditions for contraction and
  incremental stability of dynamical systems with respect to
  non-Euclidean norms. First, we introduce weak pairings as a framework
  to study contractivity with respect to arbitrary norms, and
  characterize their properties. We introduce and study the sign and max
  pairings for the $\ell_1$ and $\ell_\infty$ norms, respectively. Using weak
  pairings, we establish five equivalent characterizations for
  contraction, including the one-sided Lipschitz condition for the vector
  field as well as \textcolor{black}{logarithmic norm} and Demidovich conditions for the
  corresponding Jacobian. Third, we extend our contraction framework in
  two directions: we prove equivalences for contraction of continuous
  vector fields and we formalize the weaker notion of equilibrium
  contraction, which ensures exponential convergence to an equilibrium.
  Finally, as an application, we provide (i) incremental input-to-state
  stability and finite input-state gain properties for contracting
  systems, and (ii) a general theorem about the Lipschitz interconnection
  of contracting systems, whereby the Hurwitzness of a gain matrix
  implies the contractivity of the interconnected system.
\end{abstract}

\section{Introduction}

\begin{table*}[t]\centering
	\normalsize
	\resizebox{1\textwidth}{!}{\begin{tabular}{
				p{0.20\linewidth}
				p{0.37\linewidth}
				p{0.52\linewidth}
			}
			\textcolor{black}{Log norm}-bounded & Demidovich & One-sided Lipschitz \\
			Jacobian &  condition &  condition \\
			\hline
			\rowcolor{LightGray}    & &  \\[-1ex]
			\rowcolor{LightGray}
			$ \ds \mu_{2,P^{1/2}}(\jac{f}(x))\leq b$
			& $\ds P \jac{f}(x) + \jac{f}(x)^\top  P \preceq 2 b P $
			& $\ds (x-y)^\top  P \big( f(x) - f(y) \big) \leq b \norm{x-y}{2,P^{1/2}}^2$
			\\[2ex]
			\rowcolor{White} && \\[-1ex]
			\rowcolor{White}
			$ \ds \mu_{p,R}(\jac{f}(x))\leq b$
			& $\ds (Rv \circ |Rv|^{p-2})^\top R\jac{f}(x)v \leq b\|v\|_{p,R}^p $
			& $\ds ((R(x - y)) \circ |R(x-y)|^{p-2})^\top R(f(x) - f(y)) \leq b\|x -
			y\|_{p,R}^p$
			\\[2ex]
			\rowcolor{LightGray} && \\[-1ex]
			\rowcolor{LightGray}
			$\ds \mu_{1,R}(\jac{f}(x))\leq b$
			&  $\ds \sign(Rv)^{\top} R\jac{f}(x) v\le b \norm{v}{1,R}$
			&  $\ds \sign(Rx-Ry)^{\top} R(f(x) - f(y))\le b \norm{x-y}{1,R}$
			\\[2ex]
			\rowcolor{White} && \\[-1ex]
			\rowcolor{White}
			$\ds \mu_{\infty,R}(\jac{f}(x))\leq b$
			&  $\ds     \max_{i\in \Iinfty(Rv)}\! \left(R\jac{f}(x) v\right)_i(Rv)_i
			\le b \norm{v}{\infty,R}^2$ 
			&
			$\ds\max_{i\in\Iinfty(Rx-Ry)} \! (R\left(f(x)-f(y)\right))_i ((Rx)_i-(Ry)_i) \leq b \norm{x-y}{\infty,R}^2$
			\\[2ex] \hline
			&& \\[-1ex] 
	\end{tabular}}
	\caption{Table of contraction equivalences, that is, equivalences
          between measure bounded Jacobians, Demidovich and one-sided
          Lipschitz conditions. $\map{f}{\real^n}{\real^n}$ is a
          continuously differentiable vector field with Jacobian
          $\jac{f}$. Each row contains three equivalent statements, to be
          understood for all $x,y\in\real^n$ and all $v\in\real^n$. \textcolor{black}{For $p \in [1,\infty]$, the norm $\|\cdot\|_{p,R}$ is given by $\|x\|_{p,R} = \|Rx\|_p$, where $\|\cdot\|_p$ is the $\ell_p$ norm and $\mu_{p,R}(\cdot)$ is the corresponding logarithmic norm.} The
          function $\map{\sign}{\real^n}{\{-1,0,1\}^n}$ is the entrywise
          $\sign$ function, $\circ$ is the entrywise product, the absolute
          value and power of a vector are applied entrywise. We adopt the
          shorthand $\Iinfty(v) =
          \setdef{i\in\until{n}}{|v_i|=\norm{v}{\infty}}$. The matrix $P$
          is positive definite and the matrix $R$ is
          invertible.} \label{table:Demidovich}
\end{table*}

\textit{Problem description and motivation:} A vector field is contracting
if its flow map is a contraction or, equivalently, if any two solutions
approach one another exponentially fast. Contracting systems feature
highly-ordered asymptotic behavior. First, initial conditions are
forgotten.  Second, a unique equilibrium is globally exponential stable
when the vector field is time-invariant and two natural Lyapunov functions
are automatically available (i.e., the distance to the equilibrium and the
norm of the vector field). Third, a unique periodic solution is globally
exponentially stable when the vector field is periodic; in other words,
contracting system entrain to periodic inputs. Fourth and last, contracting
systems enjoy natural robustness properties such as input-to-state
stability and finite input-state gain in the presence of (Lipschitz
continuous) unmodeled dynamics.  Because of these highly-ordered and
desirable behaviors, contracting systems are of great interest for
engineering problems.

Contraction theory aims to combine, in a unified coherent framework,
results from Lyapunov stability theory, incremental stability, fixed point
theorems, monotone systems theory, and the geometry of Banach, Riemannian
and Finsler spaces.
Classical approaches primarily study contraction with respect to the
$\ell_2$ norm for continuously differentiable vector fields.
However, recent works have shown that stability can be studied more
systematically and efficiently using non-Euclidean norms (e.g., $\ell_1$,
$\ell_\infty$ and polyhedral norms) for large classes of network systems,
including biological transcriptional systems~\cite{GR-MDB-EDS:10a},
Hopfield neural networks~\cite{HQ-JP-ZBX:01}, chemical reaction
networks~\cite{MAAR-DA:16}, traffic
networks~\cite{SC-MA:15,GC-EL-KS:15,SC:19}, multi-vehicle
systems~\cite{JM-GR-RS:19}, and coupled
oscillators~\cite{GR-MDB-EDS:13,ZA-EDS:14}.
Moreover, for large-scale systems, error analysis based on the
$\ell_{\infty}$ norm may more accurately capture the effect of bounded
perturbations.
As compared with the $\ell_2$ norm, there is only limited work on
non-Euclidean contraction theory.

It is well known that contraction with respect to the $\ell_2$ norm is
established via a test on the Jacobian of the vector field; it is also true
however that an integral (derivative-free) test on the vector field itself
is equivalent.
While some differential tests are available for non-Euclidean norms, much
less is known about the corresponding integral tests.  We note that
computing Jacobians for large-scale networks may be computationally
intensive and so derivative-free contraction tests are desirable.
In this paper, we aim to characterize differential and integral tests for
arbitrary norms, paying special attention to the $\ell_1$ and $\ell_\infty$
norms, and provide a unifying framework for differential and integral
tests.

\textit{Literature review:} Contraction mappings in dynamical systems via
\textcolor{black}{logarithmic norms} have been studied extensively and can be traced back to
Lewis \cite{DCL:49}, Demidovich \cite{BPD:61} and Krasovski\u {\i}
\cite{NNK:63}. \textcolor{black}{Logarithmic norms} and numerical methods for differential
equations have been studied by Dahlquist~\cite{GD:58} and Lozinskii
\cite{SML:58} in as early as 1958; see also the influential survey by
Str{\"o}m~\cite{TS:75}.  Finally, \textcolor{black}{logarithmic norms} were applied to control
problems by Desoer and Vidyasagar in~\cite{CAD-HH:72,CAD-MV:1975,MV:02} and
contraction theory was first introduced by Lohmiller and Slotine
in~\cite{WL-JJES:98}. Since then, numerous generalizations to contraction
theory have been proposed including partial contraction~\cite{WW-JJES:05}, \textcolor{black}{contraction of stochastic differential equations~\cite{QCP-NT-JJES:09}, contraction in differential algebraic equations~\cite{IH-GS:02}}, 
contraction on Riemannian and Finsler
manifolds,~\cite{JWSP-FB:12za,FF-RS:14}, contraction for
PDEs~\cite{ZA-YS-MA-EDS:14}, transverse contraction~\cite{IRM-JJES:14},
contraction after short transients~\cite{MM-EDS-TT:16}, weak and
semi-contraction~\cite{SJ-PCV-FB:19q}, \textcolor{black}{and $k$-contraction, i.e., contraction of $k$-dimensional bodies~\cite{CW-IK-MM:22}}.

While the work of Lohmiller and Slotine explored differential conditions
for contraction for the $\ell_2$ norm, related integral conditions have
been studied in the literature under such various names as the
\emph{one-sided Lipschitz condition} in~\cite{EH-SPN-GW:93}, the \emph{QUAD
  condition} in~\cite{WL-TC:06}, the \emph{nonlinear
  measure}~\cite{HQ-JP-ZBX:01}, the \emph{dissipative Lipschitz
  condition}~\cite{TC-PEK:05}, and \emph{incremental quadratic stability}
in~\cite{LDA-MC:13}. A related unifying concept is the \emph{logarithmic
  Lipschitz constant}, advocated in~\cite{GS:06,ZA-EDS:13}.  Moreover, the
key idea appears as early as \cite{LC-DG:76}, whereby minus the vector
field is called \emph{uniformly increasing} and in the work on
discontinuous differential equations, see~\cite[Chapter~1, page~5]{AFF:88}
and references therein. Comparisons between the Lipschitz conditions, the
QUAD condition, and contraction are detailed in~\cite{PD-MdB-GR:11}, see
also~\cite[Section~1.10, Exercise~6]{EH-SPN-GW:93}.

Tests for contraction with respect to non-Euclidean non-differentiable
norms have not been widely studied.  Early results on compartmental systems
include~\cite[Theorem~2]{HM-SK-YO:78}
and~\cite[Appendix~4]{JAJ-CPS:93}. The $\ell_1$ integral test is used to
study neural networks in~\cite{HQ-JP-ZBX:01} and traffic networks
in~\cite{GC-EL-KS:15}. Recent work~\cite{SC:19} establishes that the
$\ell_1$ and $\ell_\infty$ norms are well suited to study contraction of
monotone systems. A comprehensive understanding of connections between
differential and integral conditions for these norms is desirable.

Aminzare and Sontag first drew connections between contraction theory and
so-called ``semi-inner products" in~\cite{ZA-EDS:13}, where they give
conditions for contraction in $L^p$ spaces. Since then, they have explored
contraction with respect to arbitrary norms in~\cite{ZA-EDS:14,ZA-EDS:14b},
and~\cite{ZA:15}, where arbitrary norms are used to study synchronization
of diffusively coupled systems and contractivity of reaction diffusion
PDEs.  Most notably, in \cite[Proposition 3]{ZA-EDS:14b} necessary and
sufficient conditions for contraction are given using Deimling pairings
(see~\cite[Chapter~3]{SSD:04} for more details on Deimling pairings). This
paper builds upon these underappreciated works and underutilized
connections.

\textit{Contributions:} Our first contribution is the definition of
\emph{weak pairings} as a generalization of the classic Lumer pairings, as
introduced in~\cite{GL:61,JRG:67}. We study various properties of weak
pairings, including a useful \emph{curve norm derivative formula}
applicable to dynamical systems analysis. Additionally, we establish a key
relationship between weak pairings and \textcolor{black}{logarithmic norms}, generalizing a
result by Lumer in~\cite{GL:61}; we refer to this relationship as
\emph{Lumer's equality}.  For $\ell_p$ norms, $p\in\{1,\infty\}$, we present
and characterize novel convenient choices for weak pairings: the sign
pairing for the $\ell_1$ norm and the max pairing for the $\ell_\infty$
norm.  We argue that, due to their connection with \textcolor{black}{logarithmic norms}, weak
pairings are a broadly-applicable tool for contraction analysis.

Our second contribution is proving five equivalent characterizations of
contraction for continuously differentiable vector fields on $\real^n$ with
respect to arbitrary norms. Using the language of weak pairings, we prove
the equivalence between differential and integral tests for contraction;
this result generalizes the known $\ell_2$ norm results to
non-differentiable norms such as the $\ell_1$ and $\ell_\infty$ norms. We
show that three of the five equivalences capture the \textcolor{black}{logarithmic norm}
condition, the differential condition on the vector field (referred to as
the \emph{Demidovich condition}), and the integral condition (referred to
as the \emph{one-sided Lipschitz condition}). These results
generalize~\cite[Proposition~3]{ZA-EDS:14b} in the sense that (i) we draw
an additional connection between the \textcolor{black}{logarithmic norm} of the Jacobian and the
weak pairing and (ii) use this connection and our sign and max pairings to
write the explicit differential and integral conditions in
Table~\ref{table:Demidovich}.

Our third contribution is the extension of contraction theory to vector
fields that are only continuous.  This extension demonstrates that
contraction can be understood as a property of the vector field,
independent of its Jacobian. In other words, this extension establishes the
importance of weak pairings over classical contraction approaches based on
the Jacobian of the vector field and the stability of the linearized
system.

Our fourth contribution is the formalization of \emph{equilibrium
  contraction}, a weaker form of contraction where all trajectories
exponentially converge to an equilibrium. This notion has been explored for
example, in~\cite[Chapter~2, Theorem~22]{MV:02}
and~\cite[Theorem~1]{WW-JJES:05}.
These approaches establish global exponential convergence under two
conditions: the vector field can be factorized as $f(t,x) = A(t,x)x$ and
the \textcolor{black}{logarithmic norm} of $A(t,x)$ is uniformly negative.
Our treatment of equilibrium contraction demonstrates that these two
conditions are only sufficient, whereas we provide a necessary and
sufficient characterization based upon the one-sided Lipschitz condition.

Our fifth contribution is proving novel robustness properties of
contracting and equilibrium contracting vector fields. For strongly
contracting systems we prove incremental input-to-state stability and
provide novel input-state gain estimates. These results generalize
\cite[Theorem~A]{CAD-HH:72} and~\cite[Theorem~1]{AH-ES-DDV:15}, where vector
fields are required to have a special control-affine structure. We
additionally prove a novel result for contraction under perturbations and
are able to upper bound how far the unique equilibrium shifts.

Our sixth and final contribution is a general theorem about the
contractivity of the interconnection of contracting systems.  Motivated by
applications to large scale systems, we provide a sufficient condition for
contraction and establish optimal contraction rates for interconnected
systems. This theorem is the counterpart for contracting system of the
classic theorem about the interconnection of dissipative systems, e.g.,
see~\cite[Chapter~2]{MA-CM-AP:16}.  This treatment generalizes the results
in~\cite[Section~5]{TS:75}, \cite{GR-MDB-EDS:13},
and~\cite[Lemma~3.2]{JWSP-FB:12za}, where optimal rates of contraction are
not provided, vector fields are differentiable, and interconnections with
inputs are not studied.

\begin{arxiv}
As this document is an ArXiv technical report. As compared with its
corresponding journal article, this report additionally contains (i) additional proofs of properties of
sum-decompositions of weak pairings in Appendix~\ref{app:diagWSIP} and (ii) a
treatment of semi-contraction and subspace contraction in terms of
semi-norms and log semi-norms in Appendix~\ref{app:semi}. The treatment of semi-contraction
generalizes that in~\cite{SJ-PCV-FB:19q} by providing
derivative-free conditions for semi-contraction and subspace contraction in
terms of weak pairings and the one-sided Lipschitz condition on the vector
field.
\end{arxiv}

\textit{Paper organization:} Section~\ref{SIPs} reviews Lumer pairings and
Dini derivatives. Section~\ref{sec:WSIP} defines weak pairings and provides
explicit formulas for $\ell_p$ norms. Section~\ref{sec:contraction} proves
contraction equivalences. Section~\ref{sec:robustness} gives robustness
results for contracting vector fields. Section~\ref{sec:networks} studies
the interconnection of contracting systems. Section \ref{conclusions}
provides conclusions.

\textit{Notation}: For a time-varying dynamical system $\dot{x} = f(t,x), t \in \realnonnegative, x \in \real^n$, we denote the flow starting from initial condition $x(t_0) = x_0$ by $t \mapsto \phi(t,t_0,x_0)$. If $f$ is differentiable in $x$, we denote its Jacobian by $\jac{f}(t,x) := \frac{\partial f}{\partial x}(t,x)$. We let $I_n$ be the $n \times n$ identity matrix, $\vectorzeros[n] \in \real^n$ be the vector of all zeros, and let $\|\cdot\|_{p,R}$ be the $\ell_p$ norm weighted by an invertible matrix $R \in \real^{n \times n}$ in the sense that $\|x\|_{p,R} = \|Rx\|_p$. For symmetric $A,B \in \real^{n\times n}$, $A \preceq B$ means $B - A$ is positive semidefinite. The function $\map{\sign}{\real^n}{\{-1,0,1\}^n}$ maps each entry of the vector to its sign and zero to zero.

\section{A review of Lumer pairings and Dini derivatives}
\label{SIPs}

\subsection{Norms and Lumer pairings}
\begin{defn}[Lumer pairings \cite{GL:61,JRG:67}] \label{def:SIP}
  A \emph{Lumer pairing} on $\R^n$ is a map
  $[\cdot,\cdot]: \R^n \times \R^n \to \R$ satisfying:
  \begin{enumerate}
  \item\label{SIP1} (Additivity in first argument) $[x_1+x_2,y] = [x_1, y] + [x_2,y]$, for all $x_1,x_2,y \in \R^n$,
  \item\label{SIP2} (Homogeneity) $[\alpha x, y] = [x, \alpha y] = \alpha[x, y]$, for all $x,y \in \R^n, \alpha \in \R$,
  \item\label{SIP3} (Positive definiteness) $[x,x] > 0$, for all $x \neq \vectorzeros[n]$, and
  \item\label{SIP4} (Cauchy-Schwarz inequality)\\ $|[x,y]| \leq [x,x]^{1/2}[y,y]^{1/2}$, for all $x,y \in \R^n$.
  \end{enumerate}
\end{defn}
\begin{lemma}[Norms and Lumer pairings \cite{GL:61}]
  \begin{enumerate}
  \item If $\R^n$ is equipped with a Lumer pairing, it is also a normed space with
    norm $\|x\| := [x,x]^{1/2}$, for all $x \in \real^n$.
  \item Conversely, if $\real^n$ is equipped with norm $\|\cdot\|$, then
    there exists an (not necessarily unique) Lumer pairing on $\real^n$ compatible
    with $\|\cdot\|$ in the sense that $\|x\| = [x,x]^{1/2}$, for all $x
    \in \real^n$.
  \end{enumerate}
\end{lemma}

\begin{defn}[\textcolor{black}{Logarithmic norm}~{\cite[Section 2.2.2]{MV:02}}]\label{def:lognorm}
  Let $\|\cdot\|$ be a norm on $\R^n$ and its corresponding induced
  norm in $\real^{n\times{n}}$.  The \emph{\textcolor{black}{logarithmic norm}} of $A \in \R^{n
    \times n}$ with respect to $\norm{\cdot}{}$ is
  \begin{equation}
    \mu(A):= \lim_{h \to 0^+} \frac{\|I_n + hA\| - 1}{h}.
  \end{equation}
\end{defn}
\textcolor{black}{The logarithmic norm is also referred to as matrix measure or, in what follows, log norm.} We refer to \cite{CAD-HH:72} for a list of properties of \textcolor{black}{log norms}.


\begin{lemma}[Lumer's equality~{\cite[Lemma~12]{GL:61}}] \label{lemma:Lumer}
  Given a norm $\|\cdot\|$ on $\real^n$, a compatible Lumer pairing $[\cdot,\cdot]$,
  and a matrix $A \in \R^{n \times n}$,
  \begin{equation}
    \mu(A) = \sup_{\|x\| = 1} [Ax, x] = \sup_{x \neq \vectorzeros[n]} \frac{[Ax,x]}{\|x\|^2}.
  \end{equation}
\end{lemma}

Recall that a norm $\|\cdot\|$ on $\R^n$ is \emph{differentiable} if, for
all $x, y \in \R^n \setminus \{\vect{0}_n\}$, the following limit exists:
$$\lim_{h \to 0} \frac{\|x + hy\| - \|x\|}{h}.$$ The $\ell_p$ norm is
differentiable for $p \in {]1, \infty[}$ and not differentiable for
$p\in\{1,\infty\}$.


\begin{lemma}[G\^ateaux formula for the Lumer pairing \cite{JRG:67}] \label{eq:GateauxFormula}
  Let $\|\cdot\|$ be a norm on $\R^n$. If $\|\cdot\|$ is differentiable,
  then there exists a unique compatible Lumer pairing given by the \emph{G\^ateaux
    formula}:
  \begin{equation} \label{GdFormula}
    [x,y] = \|y\|\lim_{h \to 0}\frac{\|y + hx\| - \|y\|}{h}, \; \; \; x,y
    \in \R^n \setminus \{\vect{0}_n\}.
  \end{equation}
\end{lemma}

\begin{lemma}[Lumer pairing and \textcolor{black}{log norm} for weighted $\ell_p$ norms {\cite[Example 13.1(a)]{KD:85}}]
  \label{GdSIP}
  For $p \in {]1, \infty[}$ and $R \in \mathbb{R}^{n \times n}$ invertible,
  let $\norm{\cdot}{p,R}$, $[\cdot,\cdot]_{p,R}$ and $\mu_{p,R}(\cdot)$
  denote the weighted $\ell_p$ norm, its Lumer pairing and its \textcolor{black}{log norm},
  respectively.  Then, for $x,y\in\real^n$ and $A\in\real^{n\times{n}}$,
  \begin{gather*}
    \|x\|_{p,R} = \|Rx\|_p , \quad
    [x,y]_{p,R} = \frac{(Ry \circ |Ry|^{p-2})^\top Rx}{\|y\|_{p,R}^{p-2}}, \\
    \mu_{p,R}(A) = \max_{\|x\|_{p,R} = 1} (Rx \circ |Rx|^{p-2})^\top RAx,
  \end{gather*}
  where $\circ$ is the entrywise product, $|\cdot|$ is the entrywise
  absolute value, and $(\cdot)^p$ is the entrywise power.
\end{lemma}

\begin{corollary} \label{corr:ell2}
  For $p = 2$ and $R = P^{1/2}$ where $P = P^\top \succ 0$,
  Lemma~\ref{GdSIP} implies
  \begin{gather*}
    \|x\|_{2,P^{1/2}}^2 = x^\top P x , \quad [x,y]_{2, P^{1/2}} = x^\top Py,
    \\ \mu_{2, P^{1/2}}(A) = \max_{\|x\|_{2,P^{1/2}} = 1} x^\top A^\top Px
    = \subscr{\lambda}{max}\Big( \frac{P A P^{-1} + A^\top }{2} \Big).
  \end{gather*}
\end{corollary}

\subsection{Dini derivatives}
\begin{defn}[Upper right Dini derivative]
  The \emph{upper right Dini derivative} of a function
  $\map{\varphi}{\realnonnegative}{\real}$ is
  \begin{equation}
    D^+\varphi(t) := \limsup_{h \to 0^+} \frac{\varphi(t+h) -
      \varphi(t)}{h}.
  \end{equation}
\end{defn}
\begin{lemma}[Danskin's lemma~\cite{JMD:66}]\label{lemma:Danskin}
  Given differentiable functions $\map{f_1,\dots,f_m}{{]a,b[}}{\real}$, if
  $f(t) = \max_i f_i(t)$, then
  \begin{equation}
    D^+f(t) = \max \Bigsetdef{\frac{d}{dt}f_i(t)}{f_i(t) = f(t)}.
  \end{equation}
\end{lemma}
The following two lemmas are related to known results. We report them here
for completeness sake.
\begin{lemma}[Dini derivative of absolute value function]\label{lemma:absvalue}
  Let $x: {]a,b[} \to \R$ be differentiable. Then 
  $$D^+|x(t)| = \dot{x}(t)\sign(x(t)) + |\dot{x}(t)|\chi_{\{0\}}(x(t)),$$
  where $\chi_A(x)$ is the indicator function which is 1 when $x \in A$ and
  zero otherwise.
\end{lemma}
\begin{proof}
  Since $|x(t)| = \max\{x(t),-x(t)\}$, Lemma~\ref{lemma:Danskin} implies
  $D^+|x(t)| = \dot{x}(t)$ if $x(t) > 0, -\dot{x}(t)$ if $x(t) < 0$, and
  $|\dot{x}(t)|$ if $x(t) = 0$.
\end{proof}

\begin{lemma}[Nonsmooth \GB inequality]
  \label{lemma:nonsmooth-GB}
  Let $\map{\varphi,r}{[a,b]}{\realnonnegative}$ and
  $\map{m}{[a,b]}{\real}$ be continuous. If $D^+\varphi(t) \leq
  m(t)\varphi(t) + r(t)$ for almost every $t \in {]a,b[}$, then, for every
  $t \in [a,b]$ and for $M(t) = \int_a^t m(\tau)d\tau$,
  \begin{equation}
    \varphi(t) \leq e^{M(t)}\Big(\varphi(a) + \int_a^t r(\tau)e^{-M(\tau)}d\tau\Big).
  \end{equation}
\end{lemma}
\begin{proof}
  Let $\psi(t) = \varphi(t)e^{-M(t)} \geq 0$ for all $t \in [a,b]$. Then
  \begin{align*}
    D^+\psi(t) &\leq (D^+\varphi(t) - m(t)\varphi(t))e^{-M(t)} \leq r(t)e^{-M(t)}.
  \end{align*}
  for almost every $t \in {]a,b[}$. Note that $r(t)e^{-M(t)}$ is continuous
  and \textcolor{black}{satisfies $r(t)e^{-M(t)} \geq 0$ for all $t \in [a,b]$}.
  Then by~\cite[Appendix~A1,
    Proposition~2]{TL:10}, for every $t \in [a,b]$, we have
  \begin{align*}
    \psi(t) &\leq \psi(a) + \int_a^t r(\tau)e^{-M(\tau)}d\tau, 
  \end{align*}
  which, in turn, implies the claim.
\end{proof}
\begin{lemma}[Dini comparison lemma~{\cite[Lemma 3.4]{HKK:02}}]\label{lemma:dinicomparison}
  Consider the initial value problem $\dot{\zeta} = f(t,\zeta)$,
  $\zeta(t_0) = \zeta_0$, where $\map{f}{\realnonnegative \times
    \real}{\real}$ is continuous in $t$ and locally Lipschitz in $\zeta$,
  for all $t \geq 0$ and $\zeta \in \real$. Let ${[t_0, T[}$ be the maximal
      interval of existence for $\zeta(t)$ and let
      $\map{v}{\realnonnegative}{\real}$ be continuous and satisfy
      $$D^+v(t) \leq f(t,v(t)), \quad v(t_0) \leq \zeta_0.$$
      Then $v(t) \leq \zeta(t)$ for all $t \in {[t_0,T[}$.
\end{lemma}

\begin{lemma}[Coppel's differential inequality \cite{WAC:1965}]
  \label{Coppel} Given a continuous map $(t,x) \mapsto A(t,x) \in \R^{n \times n}$,
  any solution $x(\cdot)$ of $\dot{x} = A(t,x)x$ satisfies
  \begin{equation}
  D^{+}\|x(t)\| \leq \mu(A(t,x(t)))\|x(t)\|.
  \end{equation}
\end{lemma}
We conclude with a small useful result.
\begin{lemma}
  \label{lemma:uniqueness}
  Consider the control system $\dot{x} = f(t,x,u(t))$ with
  $\map{f}{\realnonnegative \times \real^n \times \real^k}{\real^n}$
  continuous in $(t,x,u)$. Let $\|\cdot\|_{\mcX}$ be a norm on $\real^n$ and
  $\|\cdot\|_\mcU$ be a norm on $\real^k$. If there exists $\ell\geq0$ such
  that, for each $\textcolor{black}{t \in \realnonnegative,x \in \real^n, u,v \in \real^k}$,
  $$\|f(t,x,u) - f(t,x,v)\|_{\mcX} \leq \ell\|u - v\|_{\mcU},$$ then any two
  continuously differentiable solutions $x(\cdot), y(\cdot)$ to the
  control system corresponding to continuous inputs
  $\map{u_x,u_y}{\realnonnegative}{\real^k}$ and with $x(t) = y(t)$ for some $t
  \geq 0$ satisfy
  \begin{equation}
    D^+\|x(t) - y(t)\|_{\mcX} \leq \ell\|u_x(t) - u_y(t)\|_{\mcU}.
  \end{equation}
\end{lemma}
\begin{proof}
  The result follows from the definition of  Dini derivative and by
  Taylor expansions of $x(t+h)$ and $y(t+h)$.
\end{proof}

\section{Weak pairings and calculus of non-Euclidean norms} \label{sec:WSIP}
\subsection{Weak pairings definition and properties}
We define the notion of a weak pairing which further weaken the
conditions for a pairing to be a Lumer pairing.

\begin{defn}[Weak pairing] \label{defn:WSIP}
  A \emph{weak pairing} on $\R^n$ is a map $\WSIP{\cdot}{\cdot}:
  \R^n \times \R^n \to \R$ such that the following properties hold:
  \begin{enumerate}
  \item\label{WSIP1}(Subadditivity and continuity of first argument) $\WSIP{x_1+x_2}{y} \leq \WSIP{x_1}{y} + \WSIP{x_2}{y}$, for all $x_1,x_2,y \in \R^n$ and $\WSIP{\cdot}{\cdot}$ is continuous in its first argument, 
  \item\label{WSIP3}(Weak homogeneity) $\WSIP{\alpha x}{y} = \WSIP{x}{\alpha y} = \alpha\WSIP{x}{y}$ and $\WSIP{-x}{-y} = \WSIP{x}{y}$, for all $x,y \in \R^n, \alpha \geq 0$, 
  \item\label{WSIP4}(Positive definiteness) $\WSIP{x}{x} > 0$, for all $x \neq \vectorzeros[n],$
  \item\label{WSIP5}(Cauchy-Schwarz inequality) \\ $|\WSIP{x}{y}| \leq \WSIP{x}{x}^{1/2}\WSIP{y}{y}^{1/2}$, for all $x, y \in \R^n.$
  \end{enumerate}
\end{defn}
From Definition~\ref{defn:WSIP}, any Lumer pairing is a weak pairing, but
not every weak pairing is a Lumer pairing.  When necessary, we distinguish
the symbols for Lumer pairings and weak pairings and make this clear in
Table~\ref{table:symbols}.

\begin{table}[t]
  \centering
  \normalsize
  \begin{tabular}{cc}
    Symbol & Meaning \\ \hline
    \rowcolor[HTML]{C0C0C0} 
    $\|\cdot\|_{p,R}$ & $\ell_p$ norm weighted by $R$, $\|x\|_{p,R} = \|Rx\|_p$. \\ 
    $[\cdot,\cdot]_{p,R}$ & Lumer pairing compatible with $\|\cdot\|_{p,R}$ 
    \\ 
    \rowcolor[HTML]{C0C0C0}
    $\WSIP{\cdot}{\cdot}_{p,R}$ & Weak pairing compatible with $\|\cdot\|_{p,R}$ \\ 
    $\mu_{p,R}(\cdot)$ & \textcolor{black}{Log norm} with respect to $\|\cdot\|_{p,R}$    \\ \hline
  \end{tabular}
  \vspace{0.3em}
  \caption{Table of symbols. We let $p \in [1,\infty]$ and $R \in
    \real^{n\times n}$ be invertible. If a norm, Lumer pairing, weak pairing, or \textcolor{black}{log norm} does not have a subscript, it is
    assumed to be arbitrary. If the subscript $R$ is not included, $R =
    I_n$.} \label{table:symbols}
\end{table}
\begin{theorem}[Compatibility of weak pairings with norms]
  If $\WSIP{\cdot}{\cdot}$ is a weak pairing on $\real^n$, then $\|\cdot\| =
  \WSIP{\cdot}{\cdot}^{1/2}$ is a norm. Conversely, if $\R^n$ is equipped
  with a norm $\norm{\cdot}{}$, then there exists a weak pairing (but possibly
  many) such that $\WSIP{\cdot}{\cdot} = \|\cdot\|^2$.
\end{theorem}
\begin{proof}
	First, we show that $\|\cdot\| = \WSIP{\cdot}{\cdot}^{1/2}$ is a norm. Clearly it is positive definite by property \ref{WSIP4}. For homogeneity,
	\begin{align*}
	\|\alpha x\|^2 &= \WSIP{\alpha x}{\alpha x} = \alpha^2\WSIP{x}{x} \\ \implies \quad \|\alpha x\| &= |\alpha|\WSIP{x}{x}^{1/2} = |\alpha|\|x\|,
	\end{align*}
	by weak homogeneity, property \ref{WSIP3}. Finally, regarding the triangle inequality, 
	\begin{align*}
	\|x + y\|^2 &= \WSIP{x+y}{x+y} \leq \WSIP{x}{x+y} + \WSIP{y}{x+y} \\&\leq (\|x\| + \|y\|)\|x + y\|,
	\end{align*}
	by subadditivity, property \ref{WSIP1}, and the Cauchy-Schwarz inequality, property \ref{WSIP5}. This implies that $\|x + y\| \leq \|x\| + \|y\|$. \\
	For the converse, the proof is identical to that in \cite{GL:61} since any Lumer pairing is a weak pairing. 
\end{proof}

\textcolor{black}{As a consequence, if $\WSIP{x}{y}$ is a weak pairing compatible with the norm $\|\cdot\|$, then $\WSIP{Rx}{Ry}$ is a weak pairing compatible with the $R$-weighted norm $\|\cdot\|_R$ for any invertible $R \in \real^{n \times n}$.}

We now define two desirable properties of weak pairings.

\begin{defn}[Additional weak pairing properties]\label{def:LumerDeimlingCurve}
  Let $\WSIP{\cdot}{\cdot}$ be compatible with the norm $\|\cdot\|$. Then $\WSIP{\cdot}{\cdot}$ satisfies
  \begin{enumerate}
  \item\label{def:Deimlinginequality} \emph{Deimling's inequality} if, for
    all $x,y \in \real^n$,
    \begin{equation}\label{eq:DeimlingIneq}
    \WSIP{x}{y} \leq \|y\|\lim_{h \to 0^+}\frac{\|y + hx\| - \|y\|}{h},
    \end{equation}
  \item\label{def:curvenormderivative} \emph{the curve norm derivative
    formula} if, for every differentiable $x: {]a,b[} \to \R^n$ and for
    almost every $t \in {]a,b[},$
    \begin{equation}\label{eq:curvenormderivative}
      \|x(t)\|D^+\|x(t)\| = \WSIP{\dot{x}(t)}{x(t)}.
    \end{equation}
  \end{enumerate}
\end{defn}

Note that any given weak pairing may or may not satisfy these properties.
It is essentially known that any Lumer pairing satisfies Deimling's
inequality; see Appendix~\ref{app:Deimling} Lemma~\ref{SupInfPairings}.
\begin{theorem}[Lumer's equality for weak pairings]\label{theorem:lumer}
	Let $\|\cdot\|$ be a norm on $\real^n$ with compatible weak pairing $\WSIP{\cdot}{\cdot}$ satisfying Deimling's inequality,~\eqref{eq:DeimlingIneq}. Then for all $A \in \real^{n\times n}$
	\begin{equation}
	\mu(A) = \sup_{\|x\| = 1} \WSIP{Ax}{x} = \sup_{x \neq \vectorzeros[n]} \frac{\WSIP{Ax}{x}}{\|x\|^2}.
	\end{equation}
\end{theorem}
\begin{proof}
	By Deimling's inequality, for every $x \in \real^n \setminus \{\vectorzeros[n]\}$,
	\begin{align*}
	\WSIP{Ax}{x} &\leq \|x\|\lim_{h \to 0^+} \frac{\|x + hAx\| - \|x\|}{h} \\
	&\leq \|x\|^2 \lim_{h \to 0^+} \frac{\|I_n + hA\|-1}{h} = \|x\|^2\mu(A).
	\end{align*}
	Thus, the inequality $\mu(A) \geq \sup_{x \neq \vectorzeros[n]} \frac{\WSIP{Ax}{x}}{\|x\|^2}$ holds.
	For the other inequality, for $v \neq \vectorzeros[n]$, we define $\Omega(v) = \frac{\WSIP{Av}{v}}{\|v\|^2}$ and note that for every $v \neq \vectorzeros[n]$ and $h > 0$, 
	\begin{align*}
	\|(I_n - hA)v\| &\geq \frac{1}{\|v\|}\WSIP{(I_n - hA)v}{v} \geq (1 - h\Omega(v))\|v\| \\
	& \geq (1 - h \sup\nolimits_{\|v\| = 1} \Omega(v))\|v\|,
	\end{align*}
	where the first inequality holds by Cauchy-Schwarz, the second by subadditivity, and the final one since $-h < 0$ and by weak homogeneity of the weak pairing. Moreover, note that $\sup_{\|v\| = 1} \Omega(v) \leq \|A\| \neq \infty$ by Cauchy-Schwarz for the weak pairing. Then for small enough $h > 0$, $I_n - hA$ is invertible and given by 
	\begin{multline*}
	(I_n - hA)^{-1} = I_n + hA + h^2A^2(I_n - hA)^{-1} \\
	\implies \|(I_n + hA)v\| \leq \|(I_n - hA)^{-1}v\| + h^2\|A^2(I_n - hA)^{-1}v\|, 
	\end{multline*}
	where the last implication holds for all $v \in \real^n$ because of the triangle inequality.
	Moreover, defining $x = (I_n - hA)v$, for sufficiently small $h > 0$, we have
	\begin{equation}\label{eq:inverse}
	\frac{\|(I_n - hA)^{-1}x\|}{\|x\|} = \frac{\|v\|}{\|(I_n - hA)v\|} \leq \frac{1}{1 - h\sup\nolimits_{\|v\| = 1}\Omega(v)}
	\end{equation}
	Then
	\begin{align*}
	&\mu(A) = \lim_{h \to 0^+} \!\sup_{x \neq \vectorzeros[n]} \frac{\|(I_n + hA)x\|/\|x\| - 1}{h} \\
	&\leq \lim_{h \to 0^+} \!\sup_{x \neq \vectorzeros[n]} \!\frac{\|(I_n - hA)^{-1}x\| + h^2\|A^2(I_n - hA)^{-1}x\| - \|x\|}{h\|x\|} \\
	&\leq \lim_{h \to 0^+} \!\sup_{x \neq \vectorzeros[n]} \frac{\|(I_n - hA)^{-1}x\|/\|x\| - 1}{h} \\
	&\leq \lim_{h \to 0^+} \frac{1}{h}\Big(\frac{1}{1 - h \sup\nolimits_{\|v\| = 1}\Omega(v)} - 1\Big) = \sup_{\|x\| = 1} \WSIP{Ax}{x},
	\end{align*}
	where the first line is the definition of the induced norm, the second line holds by the triangle inequality, the third line holds due to the subadditivity of the supremum, and the last line holds because the inequality in \eqref{eq:inverse} holds for all $x \neq \vectorzeros[n]$.
\end{proof}

In the following subsections, we propose weak pairings for the $\ell_p$
norms in $\real^n$, $p \in [1,\infty]$, and show that they satisfy the two
properties in Definition~\ref{def:LumerDeimlingCurve}.

\subsection{Weak pairings for differentiable norms}
Each $\ell_p$ norm for $p \in {]1,\infty[}$ is differentiable. Therefore,
the corresponding Lumer pairing is unique, given in \textcolor{black}{Lemma~\ref{eq:GateauxFormula}}, and
satisfies Deimling's inequality,~\eqref{eq:DeimlingIneq}. Thus, we pick the weak pairing to be the
unique compatible Lumer pairing from Lemma~\ref{GdSIP}. Moreover, because
of differentiability of the norm, in Appendix~\ref{app:Deimling}
Lemma~\ref{lemma:DeimlingDerivative} we show that the unique Lumer pairing
satisfies the curve norm derivative formula in
Definition~\ref{def:LumerDeimlingCurve}\ref{def:curvenormderivative}.


\subsection{Non-differentiable norms: The $\ell_1$ norm}
The $\ell_1$ norm given by $\|x\|_1 = \sum_{i = 1}^n |x_i|$ fails to be
differentiable at points where $x_i =
0$. Hence, we
propose a pairing and show that it is a Lumer pairing compatible with the $\ell_1$
norm. 
\begin{defn}[Sign pairing]
  For $R \in \R^{n \times n}$ invertible, let $\|\cdot\|_{1,R}$ be the
  weighted $\ell_1$ norm given by $\|x\|_{1,R} = \|Rx\|_{1}$. The
  \emph{sign pairing} $\WSIP{\cdot}{\cdot}_{1,R}: \R^n \times \R^n \to \R$ is defined
  by
  \begin{equation}
    \WSIP{x}{y}_{1,R} := \|y\|_{1,R}\sign(Ry)^\top Rx.
  \end{equation}
\end{defn}
\begin{lemma}
	The sign pairing is a Lumer pairing compatible with the weighted $\ell_1$ norm.
\end{lemma}
\begin{proof}
  We verify the four properties of a Lumer pairing in Definition~\ref{def:SIP}.
  Regarding property~\ref{WSIP1}, for $x_1, x_2, y \in \R^n$,
  \begin{align*}
    \WSIP{x_1 + x_2}{y}_{1,R} &= \|y\|_{1,R} \sign(Ry)^\top R(x_1 + x_2) \\&= \|y\|_{1,R}\left( \sign(Ry)^\top Rx_1 + \sign(Ry)^\top Rx_2\right) \\ &= \WSIP{x_1}{y}_{1,R} + \WSIP{x_2}{y}_{1,R}.
  \end{align*}
  Regarding property~\ref{WSIP3}, for $\alpha \in \R$,
  \begin{align*}
    \WSIP{\alpha x}{y}_{1,R} = \|y\|_{1,R}\sign(Ry)^\top R(\alpha x) = \alpha\WSIP{x}{y}_{1,R}.
  \end{align*}
  To check homogeneity in the second argument, we see that $\alpha =
  0$ is trivial, so for $\alpha \neq 0$
  \begin{align*}
    \WSIP{x}{\alpha y}_{1,R} &= \|\alpha y\|_{1,R}\sign(\alpha Ry)^\top Rx \\&= |\alpha|\frac{\alpha}{|\alpha|}\|y\|_{1,R}\sign(Ry)^\top Rx = \alpha\WSIP{x}{y}_{1,R}.
  \end{align*}
  Regarding property~\ref{WSIP4},
  \begin{align*}
    \WSIP{x}{x}_{1,R} &= \|x\|_{1,R}\sign(Rx)^\top Rx 
    = \|x\|_{1,R}^2 \geq 0.
  \end{align*}
  This also proves compatibility.  Regarding
  property~\ref{WSIP5},
  \begin{align*}
    |\WSIP{x}{y}_{1,R}| &= \WSIP{y}{y}_{1,R}^{1/2}|\sign(Ry)^\top Rx| \\ &\leq |\sign(Rx)^\top Rx|\WSIP{y}{y}_{1,R}^{1/2} = \WSIP{x}{x}_{1,R}^{1/2}\WSIP{y}{y}_{1,R}^{1/2}.
  \end{align*}
\end{proof}
Since the sign pairing is an Lumer pairing, it is a weak pairing that satisfies Deimling's inequality,~\eqref{eq:DeimlingIneq}. Finally, we separately establish
the curve norm derivative formula,~\eqref{eq:curvenormderivative}.

\begin{theorem}[$\ell_1$ curve norm derivative formula]
	Let $x:\; ]a, b[\; \to \R^n$ be differentiable. Then
	\begin{enumerate}
		\item $D^+\|x(t)\|_{1,R} = \sign(Rx(t))^\top R\dot{x}(t),$
		for almost every $t \in {]a, b[}$. 
		\item $\|x(t)\|_{1,R}D^+\|x(t)\|_{1,R} = \WSIP{\dot{x}(t)}{x(t)}_{1,R}$
		for almost every $t \in {]a, b[}.$
	\end{enumerate}
\end{theorem}
\begin{proof}
  Since $\textcolor{black}{\|x(t)\|_{1,R} = \sum_{i = 1}^{n}|(Rx(t))_i|}$, it suffices to compute
  $D^+|x_i(t)|$. Then by Lemma \ref{lemma:absvalue}:
  \begin{align*}
    &D^+\|x(t)\|_{1,R}  = \sum\nolimits_{i = 1}^n \textcolor{black}{D^+|(Rx(t))_i|} \\&= \sum\nolimits_{i = 1}^n \left((R\dot{x}(t))_i\sign((Rx)_i) + \textcolor{black}{|(R\dot{x}(t))_i|\chi_{\{0\}}((Rx(t))_i)}\right) \\
    &= \sign(Rx(t))^\top R\dot{x}(t) + \sum\nolimits_{i = 1}^n \textcolor{black}{|(R\dot{x}(t))_i|\chi_{\{0\}}((Rx(t))_i)}.
  \end{align*}
  Multiplying both sides by $\|x(t)\|_{1,R}$ gives
  \begin{align*} 
    \|x(t)\|_{1,R}&D^+\|x(t)\|_{1,R} = \textcolor{black}{\WSIP{\dot{x}(t)}{x(t)}_{1,R}} \\ &\quad + \|x(t)\|_{1,R}\sum\nolimits_{i = 1}^n \textcolor{black}{|(R\dot{x}(t))_i|\chi_{\{0\}}((Rx(t))_i)}.
  \end{align*}
  To prove both results, it suffices to show that $\sum_{i = 1}^n
  \textcolor{black}{|(R\dot{x}(t))_i|\chi_{\{0\}}((Rx(t))_i)} = 0$ for almost every $t \in
  {]a,b[}$.  If $(Rx(t))_i \neq 0$, for all $i \in \until{n}$ and for all
  $t \in {]a,b[}$, the result holds \textcolor{black}{since $\chi_{\{0\}}((Rx(t))_i) = 0$ for all $t \in {]a,b[}, i \in \until{n}$}. So suppose $(Rx(t))_i = 0$ for
  some $i$. Then either $(Rx(t))_i = 0$ for a single $t$,
  in which case the result holds. Otherwise $(Rx(t))_i =
  0$ for all $t \in I \subseteq\; ]a, b[$, where $I$ is an interval. In
          this case, by differentiability of $x$, we have
          that $(R\dot{x}(t))_i = 0$ for almost every $t \in I$, so
	  $|(R\dot{x}(t))_i|\chi_{\{0\}}((Rx)_i(t)) = 0$
	  for almost every $t \in I$ and hence almost every $t \in {]a, b[}$. 
\end{proof}

\subsection{Non-differentiable norms: The $\ell_\infty$ norm}\label{ellinf}
The $\ell_\infty$ norm given by $\|x\|_\infty = \max_{i \in \until{n}}
|x_i|$ fails to be differentiable at points where the infinity norm is
achieved in more than one index. We propose a map and show that it is a weak pairing
that satisfies the properties in Definition~\ref{def:LumerDeimlingCurve}.
\begin{defn}[Max pairing] \label{def:MaxWSIP}
  For $R \in \R^{n \times n}$ invertible, let $\|\cdot\|_{\infty,R}$ be the
  weighted $\ell_\infty$ norm given by $\|x\|_{\infty,R} =
  \|Rx\|_{\infty}$. The \emph{max pairing} $\WSIP{\cdot}{\cdot}_{\infty,R}:
  \R^n \times \R^n \to \R$ is defined by
  \begin{equation}
    \WSIP{x}{y}_{\infty,R} := \max_{i \in I_{\infty}(Ry)}(Rx)_i (Ry)_i,
  \end{equation}
  where $\Iinfty(v) =
  \setdef{j\in\until{n}}{|v_j|=\norm{v}{\infty}}.$ 
\end{defn}
\begin{table*}[t]\centering
  \normalsize
    \resizebox{1\textwidth}{!}
            {\begin{tabular}{%
	p{0.28\linewidth}%
	p{0.28\linewidth}%
	p{0.4\linewidth}%
      }
      Norm & Weak Pairing
      & Logarithmic norm
      \\
      \hline
      \rowcolor{LightGray}    &&  \\[-2ex]
      \rowcolor{LightGray}
      $ \begin{aligned}
	\norm{x}{2,P^{1/2}} = \sqrt{x^\top  P x}
      \end{aligned}$
      &
      $\begin{aligned}
	\WSIP{x}{y}_{2, P^{1/2}} &= x^\top Py
      \end{aligned}$
      &
      $\begin{aligned}
	\mu_{2,P^{1/2}}(A) &= \min\setdef{b \in \real}{A^\top P + PA \preceq 2bP} \\ &= \tfrac{1}{2}\lambda_{\max}(PAP^{-1} + A^\top)
	\\ &= \max_{\|x\|_{2,P^{1/2}}=1} x^\top PAx 
      \end{aligned}
      $
      \\[2ex]
      &&  \\[-2ex]
      $ \begin{aligned}
	&\norm{x}{p} = \Big(\sum_{i} |x_i|^p\Big)^{1/p}\!\!\!, p \in {]1,\infty[}
      \end{aligned}$
      &
      $\begin{aligned}
	\WSIP{x}{y}_{p} &= \|y\|_{p}^{2-p}(y \circ |y|^{p-2})^\top x
      \end{aligned}$
      &
      $\begin{aligned}
	\mu_{p}(A) &= \max_{\|x\|_{p} = 1} (x \circ |x|^{p-2})^\top Ax
      \end{aligned}
      $
      \\
      \rowcolor{LightGray}    &&  \\[-2ex]
      \rowcolor{LightGray}
      $ \begin{aligned}
	\norm{x}{1} &= \sum_i |x_i|
      \end{aligned}$
      &
      $\begin{aligned}
	\WSIP{x}{y}_{1} &= \|y\|_{1}\sign(y)^\top x
      \end{aligned}$
      & $\begin{aligned}
	\mu_{1}(A) &= \max_{j \in \until{n}} \Big(a_{jj} + \sum_{i \neq j} |a_{ij}|\Big) \\ &= \sup_{\|x\|_{1} = 1}
	\sign(x)^\top Ax 
      \end{aligned}$
      \\[2ex]      
      &&  \\[-2ex]      
      $ \begin{aligned}
	\norm{x}{\infty} &= \max_i |x_i|
      \end{aligned}$
      &
      $\begin{aligned}
	\WSIP{x}{y}_{\infty} &= \max_{i \in I_{\infty}(y)} x_iy_i
      \end{aligned}$
      & $\begin{aligned}
	\mu_{\infty}(A) &= \max_{i \in \until{n}} \Big(a_{ii} + \sum_{j \neq i} |a_{ij}|\Big) \\ &=\max_{\|x\|_{\infty} = 1}
	\max_{i \in I_{\infty}(x)} (Ax)_ix_i
      \end{aligned}$
      \\[2ex] \hline
      && \\[-1ex]
  \end{tabular}}
  \caption{Table of norms, weak pairings, and \textcolor{black}{log norms} for weighted
    $\ell_2$, $\ell_p$ for $p\in{]1,\infty[}$, $\ell_1$, and $\ell_\infty$
    norms.  We adopt the shorthand $\Iinfty(x) =
    \setdef{i\in\until{n}}{|x_i|=\norm{x}{\infty}}$. The matrix $P$ is
    positive definite. Only the unweighted $\ell_p$ norms, weak pairings, and
    \textcolor{black}{log norms} for $p \neq 2$ are included here since $\mu_{p,R}(A) =
    \mu_{p}(RAR^{-1})$ for any $p \in
       [1,\infty]$.} \label{table:equivalences}
\end{table*}
\begin{lemma}
  The max pairing is a weak pairing compatible with the weighted $\ell_\infty$ norm.
\end{lemma}
\begin{proof}
  We verify the four properties of a weak pairing in Definition~\ref{defn:WSIP}.
  Regarding property~\ref{WSIP1}:
  \begin{align*} 
    &\WSIP{x_1+x_2}{y}_{\infty,R} \\ &\qquad = \max_{i \in I_{\infty}(Ry)} (R(x_1 + x_2))_i(Ry)_i \\&\qquad= \max_{i \in I_{\infty}(Ry)} (Rx_1)_i(Ry)_i + (Rx_2)_i(Ry)_i \\
    &\qquad \leq \max_{i \in I_{\infty}(Ry)} (Rx_1)_i(Ry)_i + \max_{i \in I_{\infty}(Ry)} (Rx_2)_i(Ry)_i \\&\qquad= \WSIP{x_1}{y}_{\infty,R} + \WSIP{x_2}{y}_{\infty,R}.
  \end{align*}
  Further, for fixed $y \in \R^n$, the function
  $x\mapsto\WSIP{x}{y}_{\infty,R}$ is continuous since $I_{\infty}(Ry)$ is
  fixed and the max of continuous functions is continuous.  Regarding
  property~\ref{WSIP3}, for $\alpha \geq 0$,
	\begin{align*}
	\WSIP{\alpha x}{y}_{\infty,R} &= \max_{i \in I_{\infty}(Ry)} (R\alpha x)_i(Ry)_i \\&= \alpha \max_{i \in I_{\infty}(Ry)} (Rx)_i(Ry)_i = \alpha\WSIP{x}{y}_{\infty,R}, \\
	\WSIP{x}{\alpha y}_{\infty,R} &= \max_{i \in I_{\infty}(R\alpha y)} (Rx)_i(R\alpha y)_i \\&= \alpha \max_{i \in I_{\infty}(Ry)} (Rx)_i(Ry)_i = \alpha\WSIP{x}{y}_{\infty,R}, \\
	\WSIP{-x}{-y}_{\infty,R} &= \max_{i \in I_{\infty}(-Ry)} (-Rx)_i(-Ry)_i = \WSIP{x}{y}_{\infty,R}.
	\end{align*}
	Regarding property~\ref{WSIP4} 
	\begin{align*}
	\WSIP{x}{x}_{\infty,R} &= \max_{i \in I_{\infty}(Rx)} (Rx)_i(Rx)_i = \max_{i \in I_{\infty}(Rx)} \|x\|_{\infty,R}^2 \\&= \|x\|_{\infty,R}^2 \geq 0.
	\end{align*}
	This also shows that this weak pairing is compatible with the norm. Finally, regarding property~\ref{WSIP5}:
	\begin{align*}
	&|\WSIP{x}{y}_{\infty,R}| = \left|\max_{i \in I_{\infty}(Ry)} (Rx)_i(Ry)_i\right| \\\quad &\leq \left|\max_{i \in I_{\infty}(Rx)} \|y\|_{\infty,R}\|x\|_{\infty,R}\right| = \WSIP{x}{x}_{\infty,R}^{1/2}\WSIP{y}{y}_{\infty,R}^{1/2}.
	\end{align*}
\end{proof}

We postpone to Appendix \ref{A:InftyProof} the proof of the next 
lemma.

\begin{lemma}[Deimling's inequality for the max pairing] \label{lemma:inftyProperties}
  The max pairing in Definition~\ref{def:MaxWSIP} satisfies Deimling's inequality,~\eqref{eq:DeimlingIneq}.
\end{lemma}

\begin{theorem}[Derivative of $\ell_\infty$ norm along a curve]
  Let $x: {]a, b[} \to \R^n$ be differentiable. Then for all $t \in {]a,b[},$
  \begin{enumerate}
  \item $\ds D^+\|x(t)\|_{\infty,R} = \max_{i \in I_{\infty}(Rx(t))}\sign((Rx(t))_i) (R\dot{x}(t))_i \\+ \chi_{\{\vectorzeros[n]\}}(Rx(t))\|\dot{x}(t)\|_{\infty,R}$, and 
  \item 
    $\ds\|x(t)\|_{\infty,R}D^+\|x(t)\|_{\infty,R} = \WSIP{\dot{x}(t)}{x(t)}_{\infty,R}.$
  \end{enumerate}
\end{theorem}
\begin{proof}
	From Danskin's lemma, Lemma~\ref{lemma:Danskin}, $f(t) = \max \{f_1(t), \dots, f_m(t)\}$ with differentiable $f_i$ satisfies $D^+f(t) = \max\{\frac{d}{dt}f_i(t)\; | \; f_i(t) =f(t)\}$. If the functions $f_i$ are max functions themselves (e.g., absolute values in our case), a simple argument shows
	$$D^+f(t) = \max\{D^+f_i(t)\; | \; f_i(t) = f(t)\}.$$
	By the definition of $\|\cdot\|_{\infty,R}$, we have
	\begin{align*}
	D^+\|x(t)\|_{\infty,R} = \max_{i \in I_{\infty}(Rx(t))}D^+|(Rx(t))_i|.
	\end{align*}
	Then by using the property for Dini derivatives of the absolute value function as in Lemma \ref{lemma:absvalue},
	\begin{align*}
	D^+\|x(t)\|_{\infty,R} &= \max_{i \in I_{\infty}(Rx(t))} \sign((Rx(t))_i)(R\dot{x}(t))_i \\& \qquad + |(R\dot{x}(t))_i|\chi_{\{0\}}(Rx(t))_i \\
	&= \max_{i \in I_{\infty}(Rx(t))}\sign((Rx(t))_i) (R\dot{x}(t))_i \\& \qquad+ \chi_{\{\vectorzeros[n]\}}(Rx(t))\|\dot{x}(t)\|_{\infty,R}.
	\end{align*}
	This proves the first result. To get the second result, multiply both sides by
	$\|x(t)\|_{\infty,R}$ to get
	\begin{align*}
	\|x(t)\|_{\infty,R}&D^+\|x(t)\|_{\infty,R} = \WSIP{\dot{x}(t)}{x(t)}_{\infty,R} \\&\quad + \|x(t)\|_{\infty,R}\chi_{\{\vectorzeros[n]\}}(Rx(t))\|\dot{x}(t)\|_{\infty,R}.
	\end{align*}
	Note that this second term is identically zero since if $\chi_{\{\vectorzeros[n]\}}(Rx(t)) = 1,$ then $\|x(t)\|_{\infty,R} = 0$.
\end{proof}
Weak pairings, known expressions for \textcolor{black}{log norms}, and novel expressions for \textcolor{black}{log norms} from Lumer's equality for $\ell_p$~norms are summarized in Table~\ref{table:equivalences}.

\section{Contraction theory via weak pairings}\label{sec:contraction}
\subsection{One-sided Lipschitz functions}
\begin{definition}[One-sided Lipschitz function]
	Let $\map{f}{C}{\real^n}$, where $C \subseteq \real^n$ is open and connected. We say $f$ is \emph{one-sided Lipschitz with respect to a weak pairing $\WSIP{\cdot}{\cdot}$} if the weak pairing satisfies Deimling's inequality,~\eqref{eq:DeimlingIneq}, and there exists $b \in \real$ such that
	\begin{equation}
	\WSIP{f(x) - f(y)}{x - y} \leq b\|x - y\|^2 \quad \text{for all } x,y \in C.
	\end{equation}
	We say $b$ is a \emph{one-sided Lipschitz constant of $f$}. Moreover, the minimal one-sided Lipschitz constant of $f$, $\osL(f)$, is
	\begin{equation}
	\osL(f) := \sup_{x \neq y} \frac{\WSIP{f(x) - f(y)}{x - y}}{\|x - y\|^2} \in \real \cup \{\infty\}.
	\end{equation}
\end{definition}

We prove the following proposition in Appendix~\ref{app:osLProp}.
\begin{proposition}[Properties of $\osL(f)$]\label{prop:osLProperties}
	Let $\map{f,g}{C}{\real^n}$ be one-sided Lipschitz with respect to a weak pairing $\WSIP{\cdot}{\cdot}$. Then 
	for $c \in \real$ and $\textcolor{black}{\map{\id}{C}{C}}$ the identity map:
	\begin{enumerate}
		\item\label{osL1} $\osL(f) \leq \sup_{x \neq y} \frac{\|f(x) - f(y)\|}{\|x - y\|}$,
		\item\label{osL2} $\osL(f + c\id) = \osL(f) + c$,
		\item\label{osL3} $\osL(\alpha f) = \alpha \osL(f), \quad$ for all $\alpha \geq 0$,
		\item\label{osL4} $\osL(f + g) \leq \osL(f) + \osL(g)$.
	\end{enumerate}
\end{proposition}


\begin{remark}
	When $\map{f}{C}{\real^n}$ is continuously differentiable and $C$ is convex, $\osL(f)$ does not depend on the choice of weak pairing and instead depends only on the norm since
	\begin{equation*}
	\sup_{x \neq y} \frac{\WSIP{f(x) - f(y)}{x - y}}{\|x - y\|^2} = \sup_{x \in C} \mu(\jac{f}(x)),
	\end{equation*}
	which follows from the mean-value theorem for vector-valued functions in conjunction with Lumer's equality. \oprocend
\end{remark}


\subsection{Contraction equivalences for continuously differentiable vector fields} 
\begin{theorem}[Contraction equivalences for continuously differentiable vector fields]
  \label{thm:general}
  Consider the dynamics $\dot{x} = f(t,x)$, with $f$ continuously
  differentiable in $x$ and continuous in $t$. Let $C \subseteq \real^n$ be
  open, convex, and forward invariant and let $\|\cdot\|$ denote a norm with
  compatible weak pairing $\WSIP{\cdot}{\cdot}$ satisfying Deimling's
  inequality,~\eqref{eq:DeimlingIneq}. Then, for $b \in \real$, the following statements are
  equivalent:
  \begin{enumerate}
  \item\label{ctGen:5} 
  $\osL(f(t,\cdot)) \leq b$ with respect to the weak pairing $\WSIP{\cdot}{\cdot}$, for all $t \geq 0$,
  \item\label{ctGen:4} $\WSIP{\jac{f}(t,x)v}{v} \leq b\|v\|^2$, for all $v \in \R^n, x \in C, t \geq 0$,
  \item\label{ctGen:2} $\mu(\jac{f}(t,x)) \leq b$, for all $x \in C, t \geq 0$, 
  \item\label{ctGen:6} $D^+\|\phi(t,t_0,x_0) - \phi(t,t_0,y_0)\| \leq b\|\phi(t,t_0,x_0) - \phi(t,t_0,y_0)\|$, for all $x_0,y_0 \in C, 0 \leq t_0 \leq t$ for which the solutions exist,
  \item\label{ctGen:1} $\|\phi(t,t_0,x_0) - \phi(t,t_0,y_0)\| \leq e^{b(t-s)}\|\phi(s,t_0,x_0) - \phi(s,t_0,y_0)\|$, for all $x_0, y_0 \in C$ and $0 \leq t_0 \leq s \leq t$ for which the solutions exist.
	\end{enumerate}
\end{theorem}
\begin{proof}
	Regarding \ref{ctGen:5} $\implies$ \ref{ctGen:4}, if $v = \vectorzeros[n]$, the result is trivial. By definition of $\osL(f(t,\cdot))$, $\WSIP{f(t,x) -
		f(t,y)}{x-y} \leq b\|x - y\|^2$, for all $x, y \in C, t \geq 0$. Fix $y \neq x$
	and set $x = y + hv$ for an arbitrary $v \in \R^n$ and $h \in
	\realpositive$ sufficiently small. Then
	\begin{align*}
	\WSIP{f(t,y + hv) - f(t,y)}{hv} &\leq b\|hv\|^2 \\ \implies \quad h\WSIP{f(t,y + hv) - f(t,y)}{v} &\leq bh^2\|v\|^2,
	\end{align*}
	by the weak homogeneity of the weak pairing, property
	\ref{WSIP3}. Dividing by $h^2$ and taking the limit as $h$ goes to zero
	yields
	\begin{align*}
	\lim_{h \to 0^+} \WSIP{\frac{f(t,y + hv) - f(t,y)}{h}}{v} &\leq b\|v\|^2 \\  \implies \quad \WSIP{\jac{f}(t,y)v}{v} &\leq b\|v\|^2,
	\end{align*}
	which follows from the continuity of the weak pairing in its first
	argument, property \ref{WSIP1}. Since $y$, $v$, and $t$ were arbitrary,
	this completes the implication.
  
  Regarding \ref{ctGen:4} $\implies$ \ref{ctGen:2}, suppose
  $\WSIP{\jac{f}(t,x)v}{v} \leq b\|v\|^2$ for all $x \in C, v \in \real^n,
  t \geq 0$. Let $v \neq \vectorzeros[n]$ and divide by $\|v\|^2$. Then
  take the $\sup$ over all $v \neq \vectorzeros[n]$ to get
  $\mu(\jac{f}(t,x)) = \sup_{v\neq \vectorzeros[n]}
  \WSIP{\jac{f}(t,x)v}{v}/\|v\|^2 \leq b$, by Lumer's equality.


  Regarding \ref{ctGen:2} $\implies$ \ref{ctGen:6}, define $x(t) =
  \phi(t,t_0,x_0), y(t) = \phi(t,t_0,y_0)$, and $v(t) = x(t) - y(t)$ for $x_0,y_0
  \in C, t_0 \geq 0$. Then by an application of the mean-value theorem for vector-valued functions,
  $$\dot{v} = \left(\int_{0}^1 \jac{f}(t,y + sv)ds\right)v.$$
  By an application of Coppel's differential inequality, Lemma~\ref{Coppel}, we have
  \begin{align*}
    D^+\|v(t)\| &\leq \mu\left(\int_{0}^1 \jac{f}(t,y(t) + sv(t))ds\right)\|v(t)\|\\ &\leq \int_{0}^1 \mu(\jac{f}(t,y(t) + sv(t)))ds \|v(t)\| \leq b\|v(t)\|,
  \end{align*}
  which follows from the subadditivity of log norms,~\cite{CAD-HH:72}. 
  Substituting back gives the inequality.

  Regarding \ref{ctGen:6} $\implies$ \ref{ctGen:1}; this follows from an
  application of the nonsmooth \GB inequality,
  Lemma~\ref{lemma:nonsmooth-GB} on the interval $[s,t] \subseteq [t_0,t]$.
  
  Regarding \ref{ctGen:1} $\implies$ \ref{ctGen:5}, let $x_0,y_0 \in C$,
  $t_0 \geq 0$ be arbitrary. Then for $h \geq 0$,
  \begin{align*}
    &\|\phi(t_0 + h,t_0,x_0) - \phi(t_0 + h,t_0,y_0)\| \\&\qquad = \|x_0 - y_0 + h(f(t_0,x_0) - f(t_0,y_0))\| + O(h^2) \\
    &\qquad \leq e^{bh}\|x_0 - y_0\|.
  \end{align*}
  Subtracting $\|x_0 - y_0\|$ on both sides, dividing by $h > 0$ and taking the limit as $h \to 0^+$, we get
  \begin{align*}
    &\lim_{h \to 0^+} \frac{\|x_0 - y_0 + h(f(t_0,x_0) - f(t_0,y_0))\| - \|x_0 - y_0\|}{h} \\ &\qquad \leq \lim_{h \to 0^+} \frac{e^{bh} - 1}{h}\|x_0 - y_0\|.
  \end{align*}
  Evaluating the right hand side limit gives
  \begin{align}
    &\lim_{h \to 0^+} \frac{\|x_0 - y_0 + h(f(t_0,x_0) - f(t_0,y_0))\| - \|x_0 - y_0\|}{h} \nonumber \\&\qquad \leq b\|x_0 - y_0\|. \label{eq:manipulation}
  \end{align}
  But by the assumption of Deimling's inequality,
  \begin{align*}
    &\WSIP{f(t_0,x_0) - f(t_0,y_0)}{x_0-y_0} \\ &\leq \|x_0 - y_0\|\lim_{h \to 0^+} \tfrac{\|x_0 - y_0 + h(f(t_0,x_0) - f(t_0,y_0))\| - \|x_0 - y_0\|}{h}.
  \end{align*}
  Then multiplying both sides of \eqref{eq:manipulation} by $\|x_0 - y_0\|$ gives
  $$\WSIP{f(t_0,x_0) - f(t_0,y_0)}{x_0-y_0} \leq b\|x_0 - y_0\|^2.$$
  Since $t_0, x_0$ and $y_0$ were arbitrary, the result holds.
  
\end{proof}
\begin{remark}
  \begin{enumerate}[nosep]
  \item A vector field $f$ satisfying conditions~\ref{ctGen:5},
    \ref{ctGen:4} or~\ref{ctGen:2} with $b<0$ is said to be \emph{strongly
      contracting with rate $|b|$}, see \cite{WL-JJES:98}.
    Condition~\ref{ctGen:4} is referred to as the \emph{Demidovich
      condition}, see \cite{AP-AP-NVDW-HN:04}.  A system whose trajectories
    satisfy conditions~\ref{ctGen:6} or~\ref{ctGen:1} with $b<0$ is said to
    be \emph{incrementally exponentially stable \textcolor{black}{without
        overshoot}}, see \cite{DA:02}.

  \item Theorem \ref{thm:general} holds for any choice of weak pairing
    satisfying Deimling's inequality,~\eqref{eq:DeimlingIneq}, (but not necessarily the curve norm
    derivative formula). Moreover, if a weak pairing does not satisfy
    Deimling's inequality, condition~\ref{ctGen:4} still
    implies~\ref{ctGen:1} since $\sup_{\|x\| = 1}\WSIP{Ax}{x} \geq \mu(A)$
    from Theorem~\ref{theorem:lumer}.
    \oprocend
    
    
  \end{enumerate}  
\end{remark}

Contraction equivalences~\ref{ctGen:5},~\ref{ctGen:4}, and~\ref{ctGen:2}
are transcribed for the $\ell_p$ norms in Table~\ref{table:Demidovich} for
the choices of weak pairings given in the previous section.

\subsection{Contraction equivalences for continuous vector fields}
\begin{theorem}[Contraction equivalences for continuous vector fields]\label{thm:generalization}
  Consider the dynamics $\dot{x} = f(t,x),$ with $f$ continuous in $(t,x)$.
  Let $C
  \subseteq \real^n$ be open, connected, and forward invariant and let $\|\cdot\|$ denote a norm with compatible weak pairing
  $\WSIP{\cdot}{\cdot}$ satisfying
  \textcolor{black}{Deimling's inequality,~\eqref{eq:DeimlingIneq}, and the curve norm derivative formula,~\eqref{eq:curvenormderivative}}. Then, for $b \in \real$, the following statements are equivalent:
  \begin{enumerate}
  \item\label{DoNF1} 
  $\osL(f(t,\cdot)) \leq b$ with respect to the weak pairing $\WSIP{\cdot}{\cdot}$, for all $t \geq 0$,
  \item\label{DoNF2} $D^+\|\phi(t,t_0,x_0) - \phi(t,t_0,y_0)\| \leq b\|\phi(t,t_0,x_0) - \phi(t,t_0,y_0)\|$, for all $x_0,y_0 \in C, 0 \leq t_0 \leq t$ for which the solutions exist,
  \item\label{DoNF3} $\|\phi(t,t_0,x_0) - \phi(t,t_0,y_0)\| \leq e^{b(t-s)}\|\phi(s,t_0,x_0) - \phi(s,t_0,y_0)\|$, for all $x_0, y_0 \in C$ and $0 \leq t_0 \leq s \leq t$ for which the solutions exist.
  \end{enumerate}
  Moreover, if statements~\ref{DoNF1},~\ref{DoNF2}, and~\ref{DoNF3} hold,
  then solutions are unique. Finally, if
  statements~\ref{DoNF1},~\ref{DoNF2}, and~\ref{DoNF3} hold with $b<0$ and
  there exists $x^*\in\real^n$ such that $f(t,x^*)=0$ for all $t\geq0$,
  then solutions exist uniquely for all time $t\geq0$.
\end{theorem}
\begin{proof}
  Regarding \ref{DoNF1} $\implies$ \ref{DoNF3}, let $x_0,y_0 \in C, t_0
  \geq 0$. If $\phi(t,t_0,x_0) = \phi(t,t_0,y_0)$ for some $t \geq t_0$,
  then Lemma~\ref{lemma:uniqueness} implies that the result holds. So
  suppose $\phi(t,t_0,x_0) \neq \phi(t,t_0,y_0)$. Let $v(t) =
  \phi(t,t_0,x_0)$ and $w(t) = \phi(t,t_0,y_0)$ and apply the curve norm
  derivative formula to $v(t) - w(t)$:
  \begin{multline*}
    \|v(t) - w(t)\|D^+\|v(t) - w(t)\| \\= \WSIP{f(t,v(t)) - f(t,w(t))}{v(t) - w(t)}, 
  \end{multline*}
  for almost every $t \geq 0$. By the assumption of \ref{DoNF1}, dividing
  by $\|v(t) - w(t)\| \neq 0$ implies that
  $$D^+\|\phi(t,t_0,x_0) - \phi(t,t_0,y_0)\| \leq b\|\phi(t,t_0,x_0) -
  \phi(t,t_0,y_0)\|,$$ for almost every $t \geq t_0$. Then applying the
  nonsmooth \GB inequality, Lemma~\ref{lemma:nonsmooth-GB},
  gives~\ref{DoNF3}.  Regarding \ref{DoNF3} $\implies$ \ref{DoNF1}, the
  proof is the same as in Theorem~\ref{thm:general}.  Regarding \ref{DoNF2}
  $\implies$ \ref{DoNF3}, the result follows from the nonsmooth \GB
  inequality, Lemma~\ref{lemma:nonsmooth-GB}.  Regarding \ref{DoNF3}
  $\implies$ \ref{DoNF2}, we invoke a concept from
  Appendix~\ref{app:Deimling}. Following the proof of~\ref{DoNF3}
  $\implies$~\ref{DoNF1} gives the inequality $(f(t,x) - f(t,y),x-y)_+ \leq
  b\|x - y\|^2$. Then applying the curve norm derivative formula for
  Deimling pairings, Lemma~\ref{lemma:DeimlingDerivative}, with $v(t) =
  \phi(t,t_0,x_0), w(t) = \phi(t,t_0,y_0)$ for $x_0,y_0 \in C, t_0 \geq 0$,
  implies that
  \begin{multline*}
    \|v(t) - w(t)\|D^+\|v(t) - w(t)\| \\\leq (f(t,v(t)) - f(t,w(t)), v(t) - w(t))_+,
  \end{multline*}
  for all $t \geq t_0$ for which $v(t),w(t)$ exist. Then substituting the
  previous inequality gives the result.  To see uniqueness, note that if
  $\phi(t_0,t_0,x_0) = \phi(t_0,t_0,y_0)$, then $\|x_0 - y_0\| = 0$ and
  $\|\phi(t,t_0,x_0) - \phi(t,t_0,y_0)\| = 0$ for all $t \geq t_0$ for
  which the solutions exist by \ref{DoNF3}. Regarding existence, if $b <
  0$, consider the flow $\phi(t,t_0,x^*)$ which is constant for all $t \geq
  t_0$. Then any other solution exponentially converges to $x^*$ and must
  exist for all $t \geq 0$.
\end{proof}

\begin{remark}
  If $f(t,x)$ is only piecewise continuous in $t$, then both~\ref{DoNF1}
  and~\ref{DoNF2} imply~\ref{DoNF3}, but the converse need not hold. We
  refer to \cite{MDB-DL-GR:14} for contraction results for piecewise smooth
  vector fields in terms of their
  Jacobians. Theorem~\ref{thm:generalization} does not require the
  computation of Jacobians and demonstrates that contraction is completely
  captured by the one-sided Lipschitz condition.  \oprocend
\end{remark}

\subsection{Equilibrium contraction}

\begin{theorem}[Equilibrium contraction theorem]\label{aveContr}
  Consider the dynamics $\dot{x} = f(t,x)$, with $f$ continuous in $(t,x)$. Assume there exists
  $x^*$ satisfying $f(t,x^*) = \vectorzeros[n]$ for all $t \geq 0$. Let $C
  \subseteq \real^n$ be open, connected, and forward invariant with
  $x^* \in C$ and let $\|\cdot\|$ denote a norm with compatible weak pairing
  $\WSIP{\cdot}{\cdot}$ satisfying \textcolor{black}{Deimling's inequality,~\eqref{eq:DeimlingIneq}, and the curve norm derivative formula,~\eqref{eq:curvenormderivative}}. Then, for $b \in \real$, the following statements are equivalent:
  \begin{enumerate}
  \item \label{ave3} $\WSIP{f(t,x)}{x-x^*} \leq b\|x - x^*\|^2$, for all $x \in C, t \geq 0,$
  \item \label{ave4} $D^+\|\phi(t,t_0,x_0) - x^*\| \leq b\|\phi(t,t_0,x_0) - x^*\|$, for all $x_0 \in C, 0 \leq t_0 \leq t$. 
  \item \label{ave5} $\|\phi(t,t_0,x_0) - x^*\| \leq e^{b(t-s)}\|\phi(s,t_0,x_0) - x^*\|$, for all $x_0 \in C, 0 \leq t_0 \leq s \leq t$.
  \end{enumerate}
  Moreover, if $C$ is convex and there exists a continuous map $(t,x) \mapsto A(t,x) \in
  \real^{n \times n}$ such that $f(t,x) = A(t,x)(x - x^*)$ for all $t,x$,
  then $\mu(A(t,x)) \leq b$ for all $t,x$ implies \ref{ave3}, \ref{ave4},
  and \ref{ave5}.
\end{theorem}
\begin{proof}
  Regarding \ref{ave3} $\implies$ \ref{ave5}, let $x_0 \in C, t_0 \geq 0$. If $\phi(t,t_0,x_0) = x^*$, then the
  result holds. So assume $\phi(t,t_0,x_0) \neq x^*$, let $v(t) = \phi(t,t_0,x_0)$ and apply the curve norm
  derivative formula to $v(t) - x^*$. Then
  \begin{align*}
    \|v(t) - x^*\|D^+\|v(t) - x^*\| &= \WSIP{f(t,v(t))}{v(t) - x^*} \\&\leq b\|v(t) - x^*\|^2,
  \end{align*}
  for almost every $t \geq 0$. Dividing by $\|v(t) - x^*\| \neq 0$ implies
  $D^+\|v(t) - x^*\| \leq b\|v(t) - x^*\|$ for almost every $t \geq
  0$. Then the nonsmooth \GB inequality,
  Lemma~\ref{lemma:nonsmooth-GB}, gives \ref{ave5}.

  Regarding \ref{ave4} $\implies$ \ref{ave5}, this result follows by
  applying the nonsmooth \GB inequality in Lemma~\ref{lemma:nonsmooth-GB}.

  Regarding \ref{ave5} $\implies$ \ref{ave3}, let $x_0 \in C, t_0 \geq 0$. Then for every $h > 0$,
  \begin{align*}
    \|\phi(t_0+h,t_0,x_0) - x^*\| &= \|x_0 - x^* + h(f(t_0,x_0))\| + O(h^2) \\&\leq e^{bh}\|x_0 - x^*\|.
  \end{align*}
  Subtracting $\textcolor{black}{\|x_0 - x^*\|}$ on both sides, dividing by $h > 0$ and taking the limit as $h \to 0^+$, we get
  \begin{align*}
    &\lim_{h \to 0^+} \frac{\|x_0 - x^* + h(f(t_0,x_0))\| - \|x_0 - x^*\|}{h} \\& \qquad\qquad \leq \lim_{h \to 0^+} \frac{e^{bh} - 1}{h}\|x_0 - x^*\|.
  \end{align*}
  Evaluating the right hand side limit and multiplying both sides by $\|x_0 - x^*\|$ gives
  \begin{align*}
    &\|x_0 - x^*\|\lim_{h \to 0^+}\frac{\|x_0 - x^* + h(f(t_0,x_0))\| - \|x_0 - x^*\|}{h} \\& \qquad \leq b\|x_0 - x^*\|^2,
  \end{align*}
  However, by the assumption of the weak pairing satisfying Deimling's inequality,~\eqref{eq:DeimlingIneq}, we get
  $\WSIP{f(t_0,x_0)}{x_0-x^*} \leq b\|x_0 - x^*\|^2.$
  Since $t_0, x_0$ were arbitrary, the result holds.

  Regarding \ref{ave5} $\implies$ \ref{ave4}, we invoke a concept from
  Appendix~\ref{app:Deimling}. Following the proof of \ref{ave5} $\implies$
  \ref{ave3}, we have
  \begin{equation} \label{eq:AverageDeimling}
    (f(t,x), x - x^*)_+ \leq b\|x - x^*\|^2, \quad \text{for all } x \in C, t \geq 0.
  \end{equation}
  Then let $x_0 \in C, t_0 \geq 0$, let $v(t) = \phi(t,t_0,x_0)$, and apply
  the curve norm derivative formula for Deimling pairings, 
  Lemma~\ref{lemma:DeimlingDerivative}, to $v(t) - x^*$ to get
  $$\|v(t) - x^*\|D^+\|v(t) - x^*\| \leq (f(t,v(t)),v(t))_+,$$ for all $t
  \geq 0$.  Using the inequality in \eqref{eq:AverageDeimling} proves the
  result.

  Now suppose that there exists a continuous map $(t,x) \mapsto
  A(t,x)$ such that $f(t,x) = A(t,x)(x-x^*)$ and $\mu(A(t,x)) \leq b$ for
  all $x \in C$ and all $t \geq 0$. Let $x_0 \in C, t_0 \geq 0$ and let $v(t) = \phi(t,t_0,x_0) - x^*$. Then
  $$\dot{v} = A(t,\phi(t,t_0,x_0))v.$$
  Applying Coppel's differential inequality, Lemma \ref{Coppel}, implies
  $$D^+\|v(t)\| \leq \mu(A(t,\phi(t,t_0,x_0)))\|v(t)\|.$$ Substituting $\mu(A(t,\phi(t,t_0,x_0)))
  \leq b$ gives \ref{ave4}.
\end{proof}

\begin{remark}
  \begin{enumerate}
  \item A vector field $f$ satisfying condition~\ref{ave3}, with $b<0$ is
    said to be \emph{equilibrium contracting with respect to $x^*$ and with
      rate $|b|$}.

  \item If $f(t,x)$ is continuously differentiable in $x$, the mean
    value theorem for vector-valued functions implies
    \begin{align*}
      f(t,x) &= f(t,x) - f(t,x^*) \\&= \left(\int_0^1 \jac{f}(t,x^* + (x -
      x^*)s)ds\right)(x - x^*).
    \end{align*}    
    One can then define the \emph{average Jacobian} of $f(t,x)$ about the
    equilibrium $x^*$ to be
    \begin{equation}
      \jacave{f}_{x^*}(t,x) := \int_0^1 \jac{f}(t,x^* + (x - x^*)s)ds.
    \end{equation}
    Therefore, there always exists at least one matrix-valued map $A(t,x)$
    such that $f(t,x)=A(t,x)(x-x^*)$.
    
  \item Condition~\ref{ave4} implies that the choice of Lyapunov function
    $V(x) = \|x - x^*\|$ for $b < 0$ gives global exponential stability
    within $C$.  
    \oprocend
  \end{enumerate}
\end{remark}

\begin{example}[Counterexample]
  To see that \ref{ave3} need not imply $\mu(A(t,x)) \leq b$, consider the
  dynamics in $\real^2$
  \begin{equation}\label{eq:exampleaverage}
    \dot{x} = A(x)x = \begin{bmatrix} -x_2^2-1 & 0 \\ 0 &
      x_1^2-1 \end{bmatrix}\begin{bmatrix} x_1 \\ x_2 \end{bmatrix},
  \end{equation}
  with equilibrium point $x^* = \vectorzeros[2]$. For the unweighted
  $\ell_2$ norm, $x\mapsto A(x)x$ satisfies
  Theorem~\ref{aveContr}\ref{ave3} with $b = -1$ since $$\WSIP{A(x)x}{x}_2 = x^\top A(x)x =
  -\|x\|_2^2.$$ However, $\mu_2(A(x))
  = x_1^2 - 1 \geq -1$. \oprocend
\end{example}


\section{Robustness of contracting systems}\label{sec:robustness}

We include a brief review of signal norms and system gains and refer the
reader to~\cite[Chapter 2]{CAD-MV:1975} for more details. 

\begin{defn}[Signal norms and system gains~{\cite[Chapter~2]{CAD-MV:1975}}]
  Given a norm $\|\cdot\|_\mcX$ on $\mcX=\real^n$, let
  $\mathcal{L}_{\mcX}^q$, $q \in [1,\infty]$, denote the vector space of
  continuous signals $\map{x}{\realnonnegative}{\real^n}$ with well-defined
  and bounded norm
  \begin{equation}
    \|x(\cdot)\|_{\mcX,q} := \begin{cases} \Big(\int_0^\infty
      \|x(t)\|_\mcX^q dt\Big)^{1/q}, \quad & q \neq \infty, \\ \sup_{t \geq 0}
      \|x(t)\|_\mcX, & q = \infty.
    \end{cases}
  \end{equation}
  A dynamical system with state $x\in\mcX=\real^n$ and input
  $u\in\mcU=\real^k$ has \emph{$\mathcal{L}_{\mcX,\mcU}^q$ gain} bounded by
  $\gamma>0$ if, for all $u \in \mathcal{L}_{\mcU}^q$, the state $x$ from
  zero initial condition satisfies
  \begin{equation*}
    \|x(\cdot)\|_{\mcX,q} \leq  \gamma\|u(\cdot)\|_{\mcU,q}.
  \end{equation*}
\end{defn}

In what follows, for a control system $\dot{x} = f(t,x,u(t))$, we write
$x(t)$ for the flow $\phi(t,t_0,x_0)$ subject to the vector field resulting
from control input $u_x(t)$.

\begin{theorem}[Input-to-state stability and gain of contracting systems]
  \label{thm:ISS-contracting}
  For a time and input-dependent vector field $f$, consider the dynamics
  \begin{equation}\label{eq:contracting+input}
  \dot{x} = f(t,x,u(t)), \qquad x(0)=x_0\in\mcX=\real^n,
  \end{equation}
  where $u$ takes values in $\mcU=\real^k$. Assume there exists a norm
  $\map{\norm{\cdot}{\mcX}}{\real^n}{\realnonnegative}$ with compatible
  weak pairing satisfying the curve norm derivative formula,~\eqref{eq:curvenormderivative}, for all time
  $\WSIP{\cdot}{\cdot}_\mcX$, a norm
  $\map{\norm{\cdot}{\mcU}}{\real^k}{\realnonnegative}$, and positive
  scalars $c$ and $\ell$ such that
  \begin{enumerate}[label=\textup{(A\arabic*)}]
  \item\label{ass:iss:contract} 
  $\osL(f(t,\cdot,u)) \leq -c$ with respect to the weak pairing 
  $\WSIP{\cdot}{\cdot}_{\mcX}$, for all $t \geq 0$,
    $u\in\real^k$,
  \item\label{ass:iss:lipsch} $\norm{f(t,x,u)-f(t,x,v)}{\mcX}\leq \ell
    \norm{u-v}{\mcU}$, for all $t\in\realnonnegative$, $x\in\real^n$,
    $u,v\in\real^k$.
  \end{enumerate}
  Then
  \begin{enumerate}
  \item\label{eq:D+error} any two solutions $x(t)$ and $y(t)$
    to~\eqref{eq:contracting+input} with continuous input signals
    $\map{u_x,u_y}{\realnonnegative}{\real^k}$ satisfy for all $t \geq 0$,
    \begin{equation*}
      D^+\|x(t) - y(t)\|_\mcX \leq -c\|x(t) - y(t)\|_\mcX + \ell
      \norm{u_x(t)-u_y(t)}{\mcU}.
    \end{equation*}
    
  \item\label{eq:D+error-int} $f$ is \emph{incrementally input-to-state
    stable}, in the sense that, from any initial conditions
    $x_0,y_0\in\real^n$,
    \begin{multline*}
      \|x(t) - y(t)\|_\mcX \leq e^{-ct}\|x_0 - y_0\|_\mcX \\
      + \frac{\ell(1-e^{-ct})}{c} \sup_{\tau\in[0,t]}\norm{u_x(\tau)-u_y(\tau)}{\mcU},
    \end{multline*}

  \item\label{eq:incremental-gain} $f$ has \emph{incremental
    $\mathcal{L}_{_\mcX,_\mcU}^q$ gain bounded by $\ell/c$, for
    $q\in[1,\infty]$}, in the sense that solutions with $x(0)=y(0)$ satisfy
    \begin{equation}
      \norm{x(\cdot)-y(\cdot)}{\mcX,q} \leq \frac{\ell}{c} \,
      \norm{u_x(\cdot)-u_y(\cdot)}{\mcU,q}.
    \end{equation}
    \setcounter{saveenum}{\value{enumi}}
  \end{enumerate}
Next, assume that $f$ satisfies the weaker
Assumptions~\ref{ass:iss:contract-ave}-\ref{ass:iss:lipsch} instead
of~\ref{ass:iss:contract}-\ref{ass:iss:lipsch}, where
\begin{enumerate}[label=\textup{(A\arabic*$'$)}]\setcounter{enumi}{0}
\item\label{ass:iss:contract-ave} there exists $x^* \in \real^n$ and
  continuous $\map{u^*}{\realnonnegative}{\real^k}$ such that
  $f(t,x^*,u^*(t))=\vect{0}_n$ for all $t$, and
  $\WSIP{f(t,x,u)-f(t,x^*,u)}{x-y}_\mcX \leq -c \norm{x-x^*}{\mcX}^2$ for
  all $t \geq 0, x \in \real^n, u \in \real^k$.
\end{enumerate}
Then
\begin{enumerate} \setcounter{enumi}{\value{saveenum}}
\item\label{eq:ave+Error} the solution $x(t)$
  to~\eqref{eq:contracting+input} satisfies $\ds D^+\|x(t)- x^*\|_\mcX \leq
  -c\|x(t)- x^*\|_\mcX + \ell \norm{u(t)-u^*(t)}{\mcU}$, for all $t \geq
  0$,
\item\label{eq:ave+Error-int} $f$ is \emph{input-to-state stable} in the
  sense that $\ds \|x(t) - x^*\|_\mcX \leq e^{-ct}\|x_0 - x^*\|_\mcX +
  \frac{\ell(1-e^{-ct})}{c}
  \sup_{\tau\in[0,t]}\norm{u(\tau)-u^*(\tau)}{\mcU}$,
\item\label{eq:finite-gain} $f$ has \emph{ $\mathcal{L}_{\mcX,\mcU}^q$ gain
  bounded by $\ell/c$, for $q\in[1,\infty]$}.
\end{enumerate}
\end{theorem}

\begin{proof}
  The result holds at times $t \geq 0$ when $x(t) = y(t)$ by
  Lemma~\ref{lemma:uniqueness}. So assume $x(t) \neq y(t)$. By the curve
  norm derivative formula, Assumptions~\ref{ass:iss:contract}
  and~\ref{ass:iss:lipsch} imply
  \begin{align*}
    &\|x(t) - y(t)\|_\mcX D^+\|x(t) - y(t)\|_\mcX \\
    &\quad= \WSIP{f(t,x(t),u_x(t)) - f(t,y(t),u_y(t))}{x(t)-y(t)}_\mcX\\
    &\quad\leq
    \WSIP{f(t,x(t),u_x(t))- f(t,y(t),u_x(t))}{x(t)-y(t)}_\mcX \\
    &\quad\quad+  \WSIP{f(t,y(t),u_x(t)) - f(t,y(t),u_y(t))}{x(t)-y(t)}_\mcX \\
    &\quad\leq
    -c \norm{x(t)-y(t)}{\mcX}^2 + \ell\norm{x(t)-y(t)}{\mcX}\norm{u_x(t)-u_y(t)}{\mcU},
  \end{align*}
  which follows from the subadditivity and the Cauchy-Schwarz
  inequality for the weak pairing. This proves statement~\ref{eq:D+error}.
  Statement~\ref{eq:D+error-int} follows from the nonsmooth \GB inequality,
  Lemma~\ref{lemma:nonsmooth-GB}. Regarding
  statement~\ref{eq:incremental-gain}, let $w(t)=\|u_x(t)-u_y(t)\|_\mcU$
  and consider the scalar equation
  \begin{equation}\label{eq:LTI}
    \dot{\zeta} = -c\zeta + \ell w, \quad \zeta(0) = \|x(0) - y(0)\|_\mcX = 0.
  \end{equation}  
  Then by the Dini comparison lemma, Lemma~\ref{lemma:dinicomparison},
  $\|x(t) - y(t)\|_\mcX \leq \zeta(t)$ for all $t \geq 0$. Let $G$ be the
  linear operator given by $w \mapsto \zeta$ via the solution of
  $\eqref{eq:LTI}$. In other words, $\zeta = Gw$. Since $G$ arises from a
  first-order \textcolor{black}{scalar} linear system, its induced $\mathcal{L}^q$ norm is $\ell/c$
  for all $q \in [1,\infty]$, see~\cite[Proposition 2.3]{BB-MD:93}. Thus,
  \begin{equation*}
    \|x(\cdot) - y(\cdot)\|_{\mcX,q} \leq
    \frac{\ell}{c}\norm{u_x(\cdot)-u_y(\cdot)}{\mcU,q}, \; \; \text{for all } q \in [1,\infty].
  \end{equation*}
  Regarding statement~\ref{eq:ave+Error}, apply the curve norm derivative
  formula to $x(t)-x^*$ to get
  \begin{align*}
  &\|x(t) - x^*\|_\mcX D^+\|x(t) - x^*\|_\mcX \\
  &\qquad= \WSIP{f(t,x(t),u(t))-f(t,x^*,u^*(t))}{x(t)-x^*}_\mcX \\
  &\qquad\leq
  \WSIP{f(t,x(t),u(t)) - f(t,x^*,u(t))}{x(t)-x^*}_\mcX \\
  &\qquad\quad +\WSIP{f(t,x^*,u(t)) - f(t,x^*,u^*(t))}{x(t)-x^*}_\mcX \\
  &\qquad \leq -c \norm{x(t)-x^*}{\mcX}^2 +  \ell\norm{u(t) - u^*(t)}{\mcU}\norm{x(t)-x^*}{\mcX}.
  \end{align*}
  Statement~\ref{eq:ave+Error-int} then follows from the nonsmooth \GB
  inequality, Lemma~\ref{lemma:nonsmooth-GB}. The proof of
  statement~\ref{eq:finite-gain} is identical to the proof
  of~\ref{eq:incremental-gain}.
\end{proof}

The next result studies contractivity under perturbations.

\begin{theorem}[Contraction under perturbations] \label{thm:robustness-contracting}
  Consider the dynamics $\dot{x}=f(t,x) + g(t,x)$. If $\osL(f(t,\cdot)) \leq -c < 0$ and $\osL(g(t,\cdot)) \leq d \in \real$ with respect to the same weak pairing $\WSIP{\cdot}{\cdot}$ for all $t \geq 0$,
  then
  \begin{enumerate} 
  \item\label{fact:robust} \emph{(contractivity under perturbations)} if
    $d<c$, then $f+g$ is strongly contracting with rate $c-d$,
  \item\label{fact:robust:eq} \emph{(equilibrium point under
    perturbations)} if additionally $f$ and $g$ are time-invariant, then
    the unique equilibrium point $x^*$ of $f$ and $x^{**}$ of $f+g$ satisfy
    \begin{equation}
      \norm{x^*-x^{**}}{} \leq \frac{\norm{g(x^*)}{}}{c-d}.
    \end{equation}
    \setcounter{saveenum}{\value{enumi}}
  \end{enumerate}
\end{theorem}

\begin{proof}
  Statement~\ref{fact:robust} is an immediate
  consequence of Proposition~\ref{prop:osLProperties}\ref{osL4}.
  Finally, we consider the two initial value problems $\dot{x}=f(x)+g(x)$
  and $\dot{y}=f(y)+g(y)-g(x^*)$ with arbitrary initial conditions.  Note
  $f(x^*)+g(x^*)-g(x^*)=0$, that is, $x^*$ is the unique equilibrium of the
  contracting system $f(y)+g(y)-g(x^*)$.  Taking the limit as $t\to\infty$,
  Theorem~\ref{thm:ISS-contracting}\ref{eq:D+error-int} implies
  statement~\ref{fact:robust:eq}.
\end{proof}


\section{Networks of contracting systems}\label{sec:networks}
We consider the \emph{interconnection of $n$ dynamical systems}
\begin{align}\label{eq:interconnected}
  \dot{x}_i = f_i (t,x_i, x_{-i}), \quad \text{for } i\in\until{n},
\end{align}
where $x_i\in\real^{N_i}$, $N=\sum\nolimits_{i=1}^nN_i$, the subscript
{\color{black}$-i=\until{n}\setminus\{i\}$} so that
$x_{-i}\in\real^{N-N_i}$, and $\map{f_i}{\realnonnegative \times
  \real^{N_i} \times \real^{N-N_i}}{\real^{N_i}}$ is continuous.  Let
$\norm{\cdot}{i}$ and $\WSIP{\cdot}{\cdot}_i$ denote a norm and a weak
pairing on $\real^{N_i}$. Assume
\begin{enumerate}[label=\textup{(C\arabic*)}]
\item\label{A:contract} at fixed $x_{-i}$ and $t$, each map $x_i\mapsto f_i
  (t,x_i, x_{-i})$ satisfies $\textcolor{black}{\osL(f_i(t,\cdot,x_{-i})) \leq -c_i<0}$ with
  respect to $\WSIP{\cdot}{\cdot}_i$ which satisfies Deimling's inequality,~\eqref{eq:DeimlingIneq}, and the curve norm derivative formula,~\eqref{eq:curvenormderivative}. {\color{black} If $f_i$ is
    continuously differentiable in $x_i$, then this condition is equivalent to
    \begin{equation*}
      \mu_i(\jac{f_i}(t,x_i,x_{-i})) \leq -c,\enspace
      \text{for $x_i \in\real^{N_i}$, $x_{-i}\in\real^{N-N_i}$}
    \end{equation*}}
  

\item\label{A:contract:mixed} at fixed $x_i$ and $t$, each map
  $x_{-i}\mapsto f_i (t,x_i, x_{-i})$ satisfies a Lipschitz condition
  where, for all $j \in \until{n}\setminus \{i\}$, there exists
  $\gamma_{ij}\in\realnonnegative$, such that for all $x_j,y_j \in
  \real^{N_j}$,
  \begin{equation*}
    \norm{f_i (t,x_i,x_{-i}) - f_i (t,x_i,y_{-i})}{i} \leq 
    \sum_{j=1,j\neq{i}}^n    \gamma_{ij} \norm{x_j-y_j}{j}.
  \end{equation*}
\end{enumerate}

\begin{lemma}[Efficiency of diagonally weighted norm {\cite[Lemma 3]{OP-MV:06}}]\label{lemma:diagonalweights}
  Let $M \in \real^{n\times n}$ be Metzler and let $\alpha(M)$ denote its
  spectral abscissa. Then, for each $\varepsilon > 0$, there exists $\xi
  \in \realpositive^n$ satisfying the LMI
  \begin{equation}\label{eq:diagonalstability}
    \diag(\xi) M + M^\top \diag(\xi) \preceq 2 (\alpha(M) + \varepsilon) \diag(\xi).
  \end{equation}
  Moreover, if $M$ is irreducible, the result holds with $\varepsilon = 0$.
\end{lemma}

\begin{defn}[Diagonally weighted aggregation norm]
  For $\xi \in \realpositive^n$, define the $\xi$-weighted norm and
  corresponding weak pairing on $\R^N$ by
  \begin{align}
    \norm{ (x_1,\dots,x_n) }{\xi}^2 &= \sum\nolimits_{i=1}^n \xi_i \norm{x_i}{i}^2, \label{eq:diagnorm} \\
    \WSIP{(x_1,\dots,x_n)}{(y_1,\dots,y_n)}_{\xi} &= \sum\nolimits_{i=1}^n \xi_i \WSIP{x_i}{y_i}_i \label{eq:diagWSIP}.
  \end{align}
  \begin{arxiv}
    It is easy to see that \eqref{eq:diagnorm} defines a norm. We prove in
    Appendix~\ref{app:diagWSIP} that \eqref{eq:diagWSIP} is a weak pairing
    that is compatible with \eqref{eq:diagnorm} and satisfies Deimling's inequality,~\eqref{eq:DeimlingIneq} and the curve norm derivative formula,~\eqref{eq:curvenormderivative}.
  \end{arxiv}
  \begin{tac}
    It is straightforward to see that \eqref{eq:diagnorm} defines a norm
    and that~\eqref{eq:diagWSIP} defines a weak pairing that is compatible
    with~\eqref{eq:diagnorm}.
  \end{tac}
\end{defn}

\begin{theorem}[Contractivity of interconnected system]
  Consider the interconnection of continuous
  systems~\eqref{eq:interconnected} satisfying Assumptions~\ref{A:contract}
  and~\ref{A:contract:mixed} and define the \emph{gain matrix}
  \begin{equation*}
    \Gamma :=
    \begin{bmatrix}
      -c_1 &  \dots & \gamma_{1n}\\
      \vdots & & \vdots \\
      \gamma_{n1} & \dots & - c_n
    \end{bmatrix}.
  \end{equation*}
  If $\Gamma$ is Hurwitz, then
  \begin{enumerate}
  \item \label{thm:mct:1} for every $\varepsilon \in {]0,
    |\alpha(\Gamma)|[}$, there exists $\xi \in \realpositive^n$ such that
    the interconnected system is strongly contracting with respect to
    $\norm{\cdot}{\xi}$ in \eqref{eq:diagnorm} with rate $|\alpha(\Gamma) +
    \varepsilon|$, and
  \item if $\Gamma$ is irreducible, the result~\ref{thm:mct:1} holds with
    $\varepsilon = 0$.
  \end{enumerate}
\end{theorem}


\begin{proof}
  For $i\in\until{n}$, Assumptions~\ref{A:contract}
  and~\ref{A:contract:mixed} imply
  \begin{align*}
    &\WSIP{f_i (t,x_i,x_{-i}) -  f_i (t,y_i,y_{-i})}{x_i-y_i}_i \\ 
    &\qquad \leq
    \WSIP{f_i (t,x_i,x_{-i}) -  f_i (t,y_i,x_{-i})}{x_i-y_i}_i \\ &\qquad \quad +
    \WSIP{f_i (t,y_i,x_{-i}) -  f_i (t,y_i,y_{-i})}{x_i-y_i}_i
    \\
    &\qquad \leq
    - c_i \norm{x_i-y_i}{i}^2 + \sum\nolimits_{j=1,j\neq{i}}^n \gamma_{ij} \norm{x_j-y_j}{j} \norm{x_i-y_i}{i} ,
  \end{align*}
  where we used the subadditivity and Cauchy-Schwarz inequality for the
  weak pairing. By Lemma~\ref{lemma:diagonalweights}, for $\varepsilon \in
  {]0,|\alpha(\Gamma)|[}$, select $\xi \in \realpositive^n$
  satisfying~\eqref{eq:diagonalstability}.  Next, we check the one-sided
  Lipschitz condition for the interconnected system on $\real^N$ with
  respect to norm~\eqref{eq:diagnorm} and weak pairing~\eqref{eq:diagWSIP}:
  \begin{align*}
    &\sum\nolimits_{i=1}^n \xi_i
    \WSIP{f_i (t,x_i,x_{-i})   -  f_i (t,y_i,y_{-i})}{x_i-y_i}_i \\
    & \leq 
    - \sum\nolimits_{i=1}^n \xi_i c_i \norm{x_i-y_i}{i}^2 \\& \qquad \quad +
    \sum\nolimits_{i,j=1,j\neq{i}}^n \xi_i \gamma_{ij} \norm{ x_j -  y_j}{j} \norm{x_i-y_i}{i} \\
    & = 
    \begin{bmatrix}
      \norm{ x_1 -  y_1}{1} \\
      \vdots \\
      \norm{ x_n -  y_n}{n}
    \end{bmatrix}^\top  \diag(\xi) \Gamma
      \begin{bmatrix}
      \norm{ x_1 -  y_1}{1} \\
      \vdots \\
      \norm{ x_n -  y_n}{n}
      \end{bmatrix} \\
    &  = 
    \begin{bmatrix}
      \norm{ x_1 -  y_1}{1} \\
      \vdots \\
      \norm{ x_n -  y_n}{n}
    \end{bmatrix}^\top
    \frac{\diag(\xi) \Gamma + \Gamma^\top \diag(\xi)}{2}
      \begin{bmatrix}
      \norm{ x_1 -  y_1}{1} \\
      \vdots \\
      \norm{ x_n -  y_n}{n}
      \end{bmatrix},
  \end{align*}
  so that the interconnected system is strongly contracting if the gain
  matrix $\Gamma$ is diagonally stable. Moreover, using Lemma~\ref{lemma:diagonalweights}
  \begin{align*}      \sum\nolimits_{i=1}^n & \xi_i \WSIP{f_i (t,x_i,x_{-i}) - f_i
      (t,y_i,y_{-i})}{x_i-y_i}_i \\ & \leq (\alpha(\Gamma)+\varepsilon)
    \begin{bmatrix}
        \norm{ x_1 -  y_1}{1} \\
        \vdots \\
        \norm{ x_n -  y_n}{n}
      \end{bmatrix}^\top  \diag(\xi)      \begin{bmatrix}
        \norm{ x_1 -  y_1}{1} \\
        \vdots \\
        \norm{ x_n -  y_n}{n}
      \end{bmatrix} \\
      & = (\alpha(\Gamma)+\varepsilon) \sum\nolimits_{i=1}^n \xi_i \norm{ x_i -  y_i}{i}^2
      \\&  = (\alpha(\Gamma)+\varepsilon) \norm{ (x_1-y_1,\dots,x_n-y_n)}{\xi}^2.
  \end{align*}
  Then by Theorem~\ref{thm:generalization}, we have strong contraction with
  rate $|\alpha(\Gamma) + \varepsilon|$ and incremental exponential
  stability. Finally, if $\Gamma$ is irreducible, we can take $\varepsilon
  = 0$ by Lemma~\ref{lemma:diagonalweights}.
\end{proof}

An example interconnected system satisfying the
Assumptions~\ref{A:contract} and~\ref{A:contract:mixed} is of the form
$f_i(t,x_i,x_{-i}) = g_i(t,x_i) + \sum_{j=1,j\neq{i}}^n H_{ij} x_j$, where
each vector field $g_i(t,x_i)$ has one-sided Lipschitz constant $-c_i$ and
where Assumption~\ref{A:contract:mixed} is satisfied with $\gamma_{ij}$
equal to the induced gain of $H_{ij}$.

\begin{remark}[Input-to-state stability and gain of interconnected contracting systems]
  Consider interconnected subsystems of the form $\dot{x}_i =
  f(t,x_i,x_{-i},u_i)$ with an input $u_i \in \real^{k_i}$. Assume each
  $f_i$ satisfies Assumptions~\ref{A:contract} and~\ref{A:contract:mixed}
  at fixed input and, for fixed $x_i,x_{-i},t$ and all $u_i,v_i \in
  \real^{k_i}$,
  $$\|f(t,x_i,x_{-i},u_i) - f(t,x_i,x_{-i},v_i)\|_i \leq \ell_i\|u_i -
  v_i\|_{\mcU_i},$$ for some norm $\|\cdot\|_{\mcU_i}$ on $\real^{k_i}$. Then, with
  $c = |\alpha(\Gamma) + \varepsilon|$, Theorem~\ref{thm:ISS-contracting}
  shows that the interconnected system is incrementally input-to-state
  stable with
 \begin{align*}
   &\|x(t) - y(t)\|_{\xi} \leq e^{-ct}\|x_0 - y_0\|_{\xi} \\& \qquad \quad
   +\frac{1 - e^{-ct}}{c}\sum\nolimits_{i = 1}^n \ell_i\xi_i\sup_{\tau\in[0,t]}
   \|u_{x,i}(\tau) - u_{y,i}(\tau)\|_{\mcU_i},
 \end{align*}
 and has finite incremental $\mathcal{L}_{\mcX,\mcU}^q$
 gain, for any $q\in[1,\infty]$. \oprocend
\end{remark}

\section{Conclusions}\label{conclusions}
This paper presents weak pairings as a novel tool to study contraction
theory with respect to arbitrary norms. Through the language of weak
pairings, we prove contraction equivalences for continuously differentiable
vector fields, continuous vector fields, and for equilibrium contraction.
For $\ell_p$ norms with $p\in[1,\infty]$, we present explicit formulas for
the \textcolor{black}{log norms}, the Demidovich condition, and the one-sided Lipschitz
condition, leading to novel contraction equivalences for
$p\in\{1,\infty\}$.
We then prove novel robustness results for contracting and equilibrium
contracting systems including incremental input-to-state stability
properties as well as finite incremental $\mathcal{L}_{\mcX,\mcU}^q$
gain. Finally, we provide a main interconnection theorem for contracting
subsystems, that provides a counterpart to similar theorems for dissipative
subsystems.

Possible directions for future research include (i)
leveraging our non-Euclidean conditions for control design akin to control
contraction metrics~\cite{IRM-JJES:17}, (ii) studying the generalization to
nonsmooth Finsler Lyapunov functions and differentially positive systems
\cite{FF-RS:14,FF-RS:16}, (iii) exploring the additional structure 
in monotone systems~\cite{SC:19}, and finally (iv) studying generalizations
of contraction including partial contraction~\cite{WW-JJES:05},
transverse contraction~\cite{IRM-JJES:14}, and contraction after
transients~\cite{MM-EDS-TT:16}.

\section{Acknowledgments}
\textcolor{black}{The authors wish to thank Zahra Aminzare, Bassam Bamieh,
  Ian Manchester, and Anton Proskurnikov for stimulating conversations
  about contraction theory and systems theory.}

\appendices
\section{Deimling pairings}\label{app:Deimling}
\begin{defn}[Deimling pairing {\cite[Chapter 3]{KD:85}}] \label{def:Deimling}
  Given a norm $\|\cdot\|$ on $\R^n$, the Deimling pairing is the map
  $\map{(\cdot,\cdot)_+}{\real^n \times \real^n}{\real}$ defined by
  \begin{align}
    (x,y)_+ &:= \|y\|\lim_{h \to 0^+} \frac{\|y + hx\| - \|y\|}{h}.
  \end{align}
  This limit is known to exist for every $x,y \in \R^n$.
\end{defn}
If a norm is differentiable, then its associated Lumer pairing coincides
with the Deimling pairing.  The Deimling pairing is also referred to as
superior semi-inner product and right semi-inner product, \cite[Chapter
  3]{SSD:04}, \cite[Remark 1]{ZA-EDS:14b}.

\begin{lemma}[Deimling  pairing properties {\cite[Chapter 3, Proposition 5 and Corollary 5]{SSD:04}}]
	Let $\|\cdot\|$ be a norm on $\R^n$. Then the following properties hold:
	\begin{enumerate}
		\item $(x_1 + x_2, y)_+ \leq (x_1,y)_+ + (x_2,y)_+$ for all $x_1,x_2,y \in \R^n$ and $(\cdot,\cdot)_+$ is continuous in its first argument,
		\item $(\alpha x, y)_+ = (x, \alpha y)_+ = \alpha(x,y)_+$ and $(-x,-y)_+ = (x,y)_+$ for all $x,y \in \R^n, \alpha \geq 0$,
		\item $(x,x)_+ = \|x\|^2$ for all $x \in \R^n$,
		\item $|(x,y)_+| \leq \|x\|\|y\|$ for all $x, y \in \R^n$.
	\end{enumerate}
\end{lemma}

While a Deimling pairing need not be a Lumer pairing and a Lumer pairing
need not be a Deimling pairing, both Deimling pairings and Lumer pairings
are weak pairings.

\begin{lemma}[Relationship between Deimling pairing and Lumer pairings for a norm {\cite[Chapter 3, Theorem 20]{SSD:04}}] \label{SupInfPairings}
	Let $\|\cdot\|$ be a norm on $\R^n$ and let $[\cdot,\cdot]$ be a
	compatible Lumer pairing. Then
	\begin{equation}
	[x,y] \leq (x,y)_+, \quad \text{for all } x,y \in \R^n.
	\end{equation}
	Moreover, if $\subscr{\mathcal{S}}{p}$ is the set of all Lumer pairings
	compatible with the norm, then $(x,y)_+ = \sup_{[\cdot,\cdot] \in
		\subscr{\mathcal{S}}{p}} [x,y]$ for all $x,y \in \R^n$.
\end{lemma}

\begin{lemma}[Deimling curve norm derivative formula~{\cite[Proposition 13.1]{KD:85}}] \label{lemma:DeimlingDerivative}
	Let $x: {]a,b[} \to \R^n$ be differentiable. Then
	\begin{equation}
	\|x(t)\|D^+\|x(t)\| = (\dot{x}(t),x(t))_+, \quad \text{for all } t \in {]a,b[}.
	\end{equation}
\end{lemma}
Hence, any Deimling pairing satisfies Deimling's
inequality,~\eqref{eq:DeimlingIneq}, (by definition) and the curve norm
derivative formula,~\eqref{eq:curvenormderivative}, (with equality holding
for all time).

\begin{remark}[Logarithmic Lipschitz constant]
  In~\cite[Definition~5.2]{GS:06}, the least upper bound logarithmic
  Lipschitz constant of a map $\map{f}{C}{\real^n}$ is defined by
  \begin{equation}
    M^+(f) := \sup_{x \neq y} \frac{(f(x)-f(y),x-y)_+}{\|x - y\|^2}.
  \end{equation}
  In other words, $M^+(f)$ is a special case of $\osL(f)$, where $\osL$ may
  be with respect to any weak pairing satisfying Deimling's
  inequality,~\eqref{eq:DeimlingIneq}. Thus, we show that, in contraction
  theory, we are not restricted to using Deimling pairings for
  analysis. \oprocend
\end{remark}

Finally, for comparison's sake, we report from
\cite[Example~13.1(b)]{KD:85}, the Deimling pairing for the $\ell_1$ norm:
\begin{equation}\label{eq:ell1Deimling}
(x,y)_{+,1} = \|y\|_{1} \big(\sign(y)^\top x + \sum\nolimits_{i = 1}^n|x_i|\chi_{\{0\}}(y_i)\big).
\end{equation}

\section{Proof of Lemma \ref{lemma:inftyProperties}} \label{A:InftyProof}
Before we prove Lemma \ref{lemma:inftyProperties}, we first define a
class of Lumer pairings called
\emph{single-index pairings}.

\begin{lemma}[Single-index pairings]\label{lemma:singleindex}
	For $R \in \real^{n \times n}$ invertible, let $\|\cdot\|_{\infty,R}$ be
	the weighted $\ell_{\infty}$ norm. Let $2^{[n]}$ be the power set of $\until{n}$. A choice function on $\until{n}$, $f: 2^{[n]} \setminus \{\emptyset\} \to \until{n}$ satisfies  for all $S \in 2^{[n]} \setminus \{\emptyset\}$, $f(S) \in S$. Let $\subscr{\mathcal{F}}{choice}$ be the set of all choice functions on $\until{n}$. Then each $f \in \subscr{\mathcal{F}}{choice}$ defines a Lumer pairing uniquely:
	\begin{equation}
	[x,y]_{\infty,R} := (Rx)_{f(I_{\infty}(Ry))}(Ry)_{f(I_{\infty}(Ry))}.
	\end{equation}
	Moreover, each Lumer pairing is compatible with the weighted $\ell_\infty$ norm. 
\end{lemma}
\begin{proof}
	First, we prove that any choice function defines a Lumer pairing. Let $f \in \subscr{\mathcal{F}}{choice}.$ Regarding property \ref{SIP1}, let $x_1, x_2, y \in \R^n$. 
	\begin{align*}
	[x_1 + x_2, y]_{\infty,R} &= R(x_1 + x_2)_{f(I_\infty(Ry))}(Ry)_{f(I_\infty(Ry))}\\
	&= [x_1,y]_{\infty,R} + [x_2,y]_{\infty,R}.
	\end{align*}
	For property \ref{SIP2}, let $\alpha \in \R$. Then 
	\begin{align*}
	[\alpha x, y]_{\infty,R} &= (R\alpha x)_{f(I_\infty(Ry))}(Ry)_{f(I_\infty(Ry))} = \alpha[x, y]_{\infty,R}.\\
	[x, \alpha y]_{\infty,R} &= (Rx)_{f(I_\infty(R\alpha y))}(R\alpha y)_{f(I_\infty(R\alpha y))} = \alpha[x, y]_{\infty,R}.
	\end{align*}
	Regarding property \ref{SIP3}: 
	\begin{align*}
	[x,x]_{\infty,R} &= (Rx)_{f(I_\infty(Rx))}(Rx)_{f(I_\infty(Rx))} = \|x\|_{\infty,R}^2 \geq 0.
	\end{align*}
	This also proves compatibility. Finally, for property~\ref{SIP4}: 
	\begin{align*}
	|[x,y]_{\infty,R}| &= |(Rx)_{f(I_\infty(Ry))}(Ry)_{f(I_\infty(Ry))}| \\&= \|Ry\|_{\infty}|(Rx)_{f(I_\infty(Ry))}| \\ &\leq \|Ry\|_{\infty}\|Rx\|_{\infty} = [x,x]_{\infty,R}^{1/2}[y,y]_{\infty,R}^{1/2}.
	\end{align*}
\end{proof}

\begin{corollary}[Relationship between max pairing and single-index pairings]\label{cor:Maxrelationship}
	Let $\subscr{\mathcal{S}}{index}$ be the set of all single-index pairings on $\mathbb{R}^n$ compatible with norm $\|\cdot\|_{\infty,R}$. Then we have
	$$\WSIP{x}{y}_{\infty,R} \geq [x,y]_{\infty,R}, \quad \text{for all } [\cdot,\cdot]_{\infty,R} \in \subscr{\mathcal{S}}{index}, x,y \in \R^n.$$
	Moreover, for all $x,y \in \R^n$, there exists $[\cdot,\cdot]_{\infty,R} \in \subscr{\mathcal{S}}{index}$ such that
	$\WSIP{x}{y}_{\infty,R} = [x,y]_{\infty,R}.$
\end{corollary}
\begin{proof}
	$\WSIP{x}{y}_{\infty,R} \geq [x,y]_{\infty,R}$ follows by definition. Moreover, if $x,y$ are fixed, let $i^* = \arg\!\max_{i \in I_{\infty}(Ry)} (Ry)_i(Rx)_i$. Then any choice function satisfying $f(\Iinfty(Ry)) = i^*$ defines a single-index pairing with $\WSIP{x}{y}_{\infty,R} = [x,y]_{\infty,R}$.
\end{proof}

\begin{proof}[Proof of Lemma \ref{lemma:inftyProperties}]
	Let $x,y \in \real^n \setminus \{\vectorzeros[n]\}$. Then by Corollary~\ref{cor:Maxrelationship}, there exists $[\cdot,\cdot]_{\infty,R} \in \subscr{\mathcal{S}}{index}$ such that
	$\WSIP{x}{y}_{\infty,R} = [x,y]_{\infty,R}$.
	However, since $[\cdot,\cdot]_{\infty,R}$ is a Lumer pairing, it satisfies Deimling's inequality,~\eqref{eq:DeimlingIneq}. Thus,
	\begin{align*}
	\WSIP{x}{y}_{\infty,R} &= [x,y]_{\infty,R} \\&\leq \|y\|_{\infty,R}\lim_{h \to 0^+}\frac{\|y + hx\|_{\infty,R} - \|y\|_{\infty,R}}{h}.
	\end{align*}
	Since $x,y$ were arbitrary, this proves the result.
\end{proof}

\section{Proof of Proposition~\ref{prop:osLProperties}}\label{app:osLProp}

To prove Proposition~\ref{prop:osLProperties}, we will first prove one additional property of weak pairings. 
\begin{lemma}\label{cor:negativeLumer}
	Let $\|\cdot\|$ be a norm on $\real^n$ with compatible weak pairing $\WSIP{\cdot}{\cdot}$ satisfying Deimling's inequality,~\eqref{eq:DeimlingIneq}. Then for all $x \in \real^n, c \in \real$:
	\begin{equation}
	\WSIP{cx}{x} = c\|x\|^2.
	\end{equation}
\end{lemma}
\begin{proof}
	If $c \geq 0$, the result is trivial, so without loss of generality, assume $c = -1$. Lumer's equality, Theorem~\ref{theorem:lumer}, with $A = -I_n$ implies for all $x \in \real^n$
	$$\sup_{x \neq \vectorzeros[n]} \frac{\WSIP{-x}{x}}{\|x\|^2} = \mu(-I_n) = -1 \quad \implies \quad \WSIP{-x}{x} \leq -\|x\|^2.$$
	Regarding the other inequality, observe that
	\begin{align*}
	\WSIP{x}{x} &= \WSIP{x - x + x}{x} \leq \WSIP{-x}{x} + 2\WSIP{x}{x} \\ &\implies \WSIP{-x}{x} \geq -\WSIP{x}{x} = -\|x\|^2.
	\end{align*}
	By weak homogeneity, this proves the result.
\end{proof}

\begin{proof}[Proof of Proposition~\ref{prop:osLProperties}]
	Properties~\ref{osL1},~\ref{osL3}, and~\ref{osL4} are consequences of Cauchy-Schwarz, weak homogeneity, and subadditivity of the weak pairing, respectively. Regarding property~\ref{osL2}, we will show the more general result that for any $x,y \in C, c \in \real$, $\WSIP{x + cy}{y} = \WSIP{x}{y} + c\|y\|^2$. The inequality
	$$\WSIP{x + cy}{y} \leq \WSIP{x}{y} + c\|y\|^2,$$
	follows from subadditivity and Lemma~\ref{cor:negativeLumer}. Additionally,
	\begin{align*}
	\WSIP{x}{y} &= \WSIP{x + cy - cy}{y} \leq \WSIP{x + cy}{y} + \WSIP{-cy}{y} \\ &= \WSIP{x + cy}{y} - c\|y\|^2,
	\end{align*}
	where the final equality holds by Lemma~\ref{cor:negativeLumer}. Rearranging the inequality implies the result.
\end{proof}

\begin{arxiv}
\section{Sum decomposition of weak pairings}\label{app:diagWSIP}
\begin{lemma}[Weak pairing sum decomposition]\label{lemma:sumdecomp}
  For $N=\sum_{i=1}^nN_i$, let $x = (x_1,\dots,x_n), y = (y_1,\dots,y_n)
  \in \real^N$ and $x_i,y_i \in \real^{N_i}$. Let $\norm{\cdot}{i}$ and
  $\WSIP{\cdot}{\cdot}_i$ denote a norm and compatible weak pairing on
  $\real^{N_i}$ and let $\xi \in \realpositive^n$. Then the mapping
  $\map{\WSIP{\cdot}{\cdot}_\xi}{\real^N \times \real^N}{\real^N}$ defined by
	$$\WSIP{x}{y}_\xi := \sum\nolimits_{i = 1}^n \xi_i\WSIP{x_i}{y_i}_i,$$ is a weak pairing
  compatible with the norm $\norm{x}{\xi}^2 = \sum_{i=1}^n
  \xi_i\norm{x_i}{i}^2$.
\end{lemma}
\begin{proof}
	We verify the properties in Definition~\ref{defn:WSIP}. Regarding property~\ref{WSIP1}, let $x_1,x_2,y \in \real^N$. Then $\WSIP{x_1 + x_2}{y}_\xi = \sum_{i=1}^n \xi_i\WSIP{x_{1_i} + x_{2_i}}{y_i}_i \leq \sum_{i=1}^n \xi_i\WSIP{x_{1_i}}{y_i}_i + \xi_i\WSIP{x_{2_i}}{y_i}_i = \WSIP{x_1}{y}_\xi + \WSIP{x_2}{y}_\xi$. Continuity in the first argument follows from continuity of the first argument of each of the $\WSIP{\cdot}{\cdot}_i$. Regarding property~\ref{WSIP3}, the result is straightforward because of weak homogeneity of each of the $\WSIP{\cdot}{\cdot}_i$. Regarding property~\ref{WSIP4},
	$$\WSIP{x}{x}_\xi = \sum_{i = 1}^n \xi_i\WSIP{x_i}{x_i}_i = \sum_{i = 1}^n \xi_i\norm{x_i}{i}^2 > 0, \quad \text{for all } x \neq \vectorzeros[N].$$
	Regarding property~\ref{WSIP5}, since $n$ is finite, by induction it suffices to check $n = 2$. For convenience, define $a_i = \xi_i\WSIP{x_i}{x_i}_i, b_i = \xi_i\WSIP{y_i}{y_i}_i$. Then
	\begin{align*}
	\WSIP{x}{y}_\xi^2 &= \Big(\sum\nolimits_{i=1}^2 \xi_i\WSIP{x_i}{y_i}_i\Big)^2 \leq \Big(\sum\nolimits_{i=1}^2 a_i^{1/2}b_i^{1/2}\Big)^2 \\
	&= (\sqrt{a_1b_1} + \sqrt{a_2b_2})^2 = a_1b_1 + a_2b_2 + 2\sqrt{a_1b_1a_2b_2} \\
	&\leq a_1b_1 + a_2b_2 + a_1b_2 + a_2b_1 = (a_1 + a_2)(b_1 + b_2) \\
	&= \Big(\sum\nolimits_{i=1}^2 a_i\Big)\Big(\sum\nolimits_{i=1}^2 b_i\Big) = \WSIP{x}{x}_\xi\WSIP{y}{y}_\xi,
	\end{align*}
	where we have used Cauchy-Schwarz for the $\WSIP{\cdot}{\cdot}_i$ and the inequality $2\sqrt{\alpha\beta} \leq \alpha + \beta$ for $\alpha,\beta \geq 0$. Taking the square root of each side proves the result.
\end{proof}

Next we establish Deimling's inequality,~\eqref{eq:DeimlingIneq}, and the curve norm derivative formula,~\eqref{eq:curvenormderivative}.
\begin{lemma}[Deimling's inequality and curve norm derivative formula for weak pairing sum decomposition]\label{lemma:sumDeimlingCurve}
	Let $\WSIP{\cdot}{\cdot}_{\xi}$ and $\|\cdot\|_\xi$ be defined as in Lemma~\ref{lemma:sumdecomp}. If \begin{enumerate}
		\item\label{item:sumDeimling} each $\WSIP{\cdot}{\cdot}_i$ satisfies Deimling's inequality, then $\WSIP{\cdot}{\cdot}_{\xi}$ satisfies Deimling's inequality,
		\item\label{item:sumCurve} each $\WSIP{\cdot}{\cdot}_i$ satisfies the curve norm derivative formula, then $\WSIP{\cdot}{\cdot}_\xi$ satisfies the curve norm derivative formula.
	\end{enumerate}
\end{lemma}
To prove Lemma~\ref{lemma:sumDeimlingCurve}\ref{item:sumCurve}, we first prove a useful equivalent characterization of the curve norm derivative formula.

\begin{proposition}[Equivalent curve norm derivative formula characterization]\label{prop:curvenorm}
	Let $\map{x}{{]a,b[}}{\real^n}$ be differentiable and $\WSIP{\cdot}{\cdot}$ be a weak pairing compatible with the norm $\|\cdot\|$ on $\real^n$. Then the following statements are equivalent
	\begin{enumerate}
		\item\label{item:curvenormorig} $\|x(t)\|D^+\|x(t)\| = \WSIP{\dot{x}(t)}{x(t)}$ for almost every $t \in {]a,b[}$,
		\item\label{item:curvenormsquared} $D^+\|x(t)\|^2 = 2\WSIP{\dot{x}(t)}{x(t)}$ for almost every $t \in {]a,b[}$.
	\end{enumerate}
\end{proposition}
\begin{proof}
	We first prove \ref{item:curvenormorig} $\implies$ \ref{item:curvenormsquared}. Initially, suppose that $t \in {]a,b[}$ is an instant in time at which $x(t) = \vectorzeros[n]$. Then we compute
	\begin{align*}
	D^+\|x(t)\|^2 &= \limsup_{h \to 0^+} \frac{\|x(t+h)\|^2 - \|x(t)\|^2}{h} \\
	&= \lim_{h \to 0^+} \frac{\|x(t) + h\dot{x}(t)\|^2}{h} = \lim_{h \to 0^+} \frac{\|h\dot{x}(t)\|^2}{h} \\
	&= \lim_{h \to 0^+}\frac{h^2\|\dot{x}(t)\|^2}{h} = 0,
	\end{align*}
	so the result holds for all $t \in {]a,b[}$ for which $x(t) = \vectorzeros[n]$. So alternatively suppose $x(t) \neq \vectorzeros[n]$. Then 
	\begin{align*}
	&D^+\|x(t)\|^2 = \limsup_{h \to 0^+} \frac{\|x(t+h)\|^2 - \|x(t)\|^2}{h} \\
	&= \lim_{h \to 0^+} \Big(\frac{\|x(t+h)\| - \|x(t)\|}{h} (\|x(t+h)\| + \|x(t)\|)\Big) \\
	&= \big(D^+\|x(t)\|\big)\lim_{h \to 0^+} \|x(t+h)\| + \|x(t)\|, \\
	&\overset{\text{a.e.}}{=} \frac{\WSIP{\dot{x}(t)}{x(t)}}{\|x(t)\|} \cdot 2\|x(t)\| = 2\WSIP{\dot{x}(t)}{x(t)},
	\end{align*}
	where $\overset{\text{a.e.}}{=}$ denotes that the equality holds for almost every $t \in {]a,b[}$ by the assumption of \ref{item:curvenormorig}. This proves \ref{item:curvenormorig} $\implies$ \ref{item:curvenormsquared}. Regarding the other implication, note that the result holds trivially for all $t \in {]a,b[}$ for which $x(t) = \vectorzeros[n]$. Thus, we suppose that $x(t) \neq \vectorzeros[n]$ (note that this supposition implies that $\|x(t+h)\| + \|x(t)\| > 0$ for all $h > 0$). Then we compute
	\begin{align*}
	&D^+\|x(t)\| = \limsup_{h \to 0^+} \frac{\|x(t+h)\| - \|x(t)\|}{h} \\
	&= \lim_{h \to 0^+} \frac{\|x(t+h)\|^2 - \|x(t)\|^2}{h} \frac{1}{\|x(t+h)\| + \|x(t)\|} \\
	&= \big(D^+\|x(t)\|^2\big)\lim_{h \to 0^+} \frac{1}{\|x(t+h)\| + \|x(t)\|} \\
	&\overset{\text{a.e.}}{=} 2\WSIP{\dot{x}(t)}{x(t)}\frac{1}{2\|x(t)\|}.
	\end{align*}
	Multiplying both sides by $\|x(t)\|$ proves the implication.
\end{proof}
We are now ready to prove Lemma~\ref{lemma:sumDeimlingCurve}.
\begin{proof}[Proof of Lemma~\ref{lemma:sumDeimlingCurve}]
	First we prove item~\ref{item:sumDeimling}. We prove the result for $n = 2$ and then by induction the result easily extends to arbitrary $n$. For $x = (x_1,x_2)$ and $y = (y_1,y_2)$, we compute
	\begin{align*}
	&\WSIP{x}{y}_\xi = \xi_1\WSIP{x_1}{y_1}_1 + \xi_2\WSIP{x_2}{y_2}_2 \\
	&\leq \lim_{h \to 0^+} \Big(\xi_1\|y_1\|_1\frac{\|y_1 + hx_1\|_1 - \|y_1\|_1}{h} \\ &\quad + \xi_2\|y_2\|_2\frac{\|y_2 + hx_2\|_2 - \|y_2\|_2}{h}\Big) \\
	&= \lim_{h \to 0^+} \frac{\xi_1\|y_1\|_1\|y_1+hx_1\|_1 + \xi_2\|y_2\|_2\|y_2 + hx_2\|_2 - \|y\|_\xi^2}{h},
	\end{align*}
	where the first inequality holds by applying Deimling's inequality to each of $\WSIP{\cdot}{\cdot}_i$ for $i \in \{1,2\}$.
	Next we demonstrate that $\xi_1\|y_1\|_1\|y_1+hx_1\| + \xi_2\|y_2\|_2\|y_2+hx_2\|_2 \leq \|y\|_\xi \|y+hx\|_\xi$. Since both sides of the inequality are nonnegative, we square the left-hand side and compute
	\begin{align*}
	&\big(\xi_1\|y_1\|_1\|y_1+hx_1\| + \xi_2\|y_2\|_2\|y_2+hx_2\|_2\big)^2 = \\
	&\xi_1^2\|y_1\|_1^2\|y_1+hx_1\|_1^2 + \xi_2^2\|y_2\|_2^2\|y_2+hx_2\|_2^2 \\& + 2\xi_1\xi_2\|y_1\|_1\|y_2\|_2\|y_1+hx_1\|_1\|y_2+hx_2\|_2 \\
	&\leq \xi_1^2\|y_1\|_1^2\|y_1+hx_1\|_1^2 + \xi_2^2\|y_2\|_2^2\|y_2+hx_2\|_2^2 \\
	&+ \xi_1\xi_2\|y_1\|_1^2\|y_2 + hx_2\|_2^2 + \xi_1\xi_2\|y_2\|_2^2\|y_1 + hx_1\|_1^2 \\
	&= \Big(\xi_1\|y_1\|_1^2 + \xi_2\|y_2\|^2\Big)\Big(\xi_1\|y_1+hx_1\|_1^2+ \xi_2\|y_2+hx_2\|_2^2\Big) \\
	&= \|y\|_\xi^2\|y+hx\|_\xi^2,
	\end{align*}
	where the inequality holds due to $2\alpha\beta \leq \alpha^2 + \beta^2$ for all $\alpha,\beta \in \real$ with $\alpha = \|y_1\|_1\|y_2+hx_2\|_2, \beta = \|y_2\|_2\|y_2+hx_2\|_2$. This proves the desired inequality. As a consequence, we see
	\begin{align*}
	&\WSIP{x}{y}_\xi \\
	&\leq \lim_{h \to 0^+} \frac{\xi_1\|y_1\|_1\|y_1+hx_1\|_1 + \xi_2\|y_2\|_2\|y_2 + hx_2\|_2 - \|y\|_\xi^2}{h} \\
	&\leq \lim_{h \to 0^+} \frac{\|y\|_\xi\|y+hx\|_\xi - \|y\|_\xi^2}{h} \\
	&= \|y\|_\xi \lim_{h \to 0^+} \frac{\|y+hx\|_\xi - \|y\|_\xi}{h},
	\end{align*}
	which proves Deimling's inequality. \\
	Regarding item~\ref{item:sumCurve}, let $\map{x}{{]a,b[}}{\real^N}$ be differentiable. We apply Proposition~\ref{prop:curvenorm} to prove that $D^+\|x(t)\|_\xi^2 = 2\WSIP{\dot{x}(t)}{x(t)}$ for almost every $t \in {]a,b[}$. We compute
	\begin{align*}
	D^+\|x(t)\|_\xi^2 &= D^+ \Big(\sum_{i = 1}^n \xi_i \|x_i(t)\|_i^2\Big) \\
	&=\sum_{i = 1}^n \xi_i D^+\|x_i(t)\|_i^2 \overset{\text{a.e.}}{=} 2\sum_{i = 1}^n \xi_i\WSIP{\dot{x}_i(t)}{x_i(t)}_i \\&= 2\WSIP{\dot{x}(t)}{x(t)}_{\xi},
	\end{align*}
	where the third equality holds by the assumption that each $\WSIP{\cdot}{\cdot}_i$ satisfies the curve norm derivative formula. Thus, the result is proved.
\end{proof}

\end{arxiv}

\begin{arxiv}
\section{Semi-Contraction Equivalences}\label{app:semi}
We recall semi-norms and some of their properties and refer to~\cite{SJ-PCV-FB:19q} for more details.
\begin{definition}[Semi-norms]
	A map $\map{\seminorm{\cdot}}{\real^n}{\realnonnegative}$ is a \emph{semi-norm} on $\real^n$ if
	\begin{enumerate}
		\item $\seminorm{cv} = |c|\seminorm{v}$ for all $v \in \real$, $c \in \real$;
		\item $\seminorm{v + w} \leq \seminorm{v} + \seminorm{w}$, for all $v, w \in \real^n$.
	\end{enumerate}
\end{definition}

For a semi-norm $\seminorm{\cdot}$, its kernel is defined by
$$\kernel \seminorm{\cdot} = \setdef{v \in \real^n}{\seminorm{v} = 0}.$$
Note that $\kernel \seminorm{\cdot}$ is a subspace of $\real^n$ and that $\seminorm{\cdot}$ is a norm on $\kerperp{\seminorm{\cdot}}$. Since $\kerperp{\seminorm{\cdot}} \cong \real^m$ for some $m \leq n$, we are able to define weak pairings on $\kerperp{\seminorm{\cdot}}$.

\begin{definition}[Induced semi-norm]
	Let $\seminorm{\cdot}$ be a semi-norm on $\real^n$. The induced semi-norm is a map $\map{\seminorm{\cdot}}{\real^{n \times n}}{\realnonnegative}$ defined by
	$$\seminorm{A} := \sup \bigsetdef{\seminorm{Av}/\seminorm{v}}{v \in \kerperp{\seminorm{\cdot}} \setminus \{\vectorzeros[n]\}}.$$
\end{definition}
\begin{definition}[Matrix log semi-norm]
	Let $\seminorm{\cdot}$ be a semi-norm on $\real^n$ and its corresponding induced semi-norm on $\real^{n \times n}$. Then the \emph{log semi-norm} associated with $\seminorm{\cdot}$ is defined by
	$$\semimeasure{A} := \lim_{h \to 0^+} \frac{\seminorm{I_n + hA} - 1}{h}.$$
\end{definition}

We refer to~\cite[Proposition~3 and Theorem~5]{SJ-PCV-FB:19q} for properties of induced semi-norms and log semi-norms. 

We specialize to consider semi-norms of the form $\seminorm{v} = \|\mathcal{P}v\|$ where either $\mathcal{P} \in \real^{m \times n}$ is a rank $m$ matrix for $m \leq n$ and $\|\cdot\|$ is a norm on $\real^m$ or $\mathcal{P} \in \real^{n \times n}$ with rank $m \leq n$ and $\|\cdot\|$ is a norm on $\real^n$. To make this convention clear, we say $\seminorm{\cdot} = \|\mathcal{P}\cdot\|$. Then clearly $\kernel \seminorm{\cdot} = \kernel \mathcal{P}$. First, we prove a useful lemma resembling Lumer's equality, Theorem~\ref{theorem:lumer}.

\begin{lemma}[Lumer's equality for log semi-norms]\label{lemma:semiLumer}
	Let $\seminorm{\cdot} = \|\mathcal{P}\cdot\|$ be a semi-norm on $\real^n$ with compatible weak pairing $\WSIP{\cdot}{\cdot}$ on $\kerperp{\seminorm{\cdot}}$ satisfying Deimling's inequality,~\eqref{eq:DeimlingIneq}. Then for all $A \in \real^{n \times n}$ satisfying $A \kernel \seminorm{\cdot} \subseteq \kernel \seminorm{\cdot}$
	\begin{align}
	\label{eq:kerperp}\semimeasure{A} &= \sup_{x \in \kerperp{\seminorm{\cdot}} \setminus \{\vectorzeros[n]\}} \frac{\WSIP{\mathcal{P}Ax}{\mathcal{P}x}}{\seminorm{x}^2} \\
	&\label{eq:ker}= \sup_{x \notin \kernel \seminorm{\cdot}} \frac{\WSIP{\mathcal{P}Ax}{\mathcal{P}x}}{\seminorm{x}^2}.
	\end{align}
\end{lemma}
\begin{proof}
	First we show the equivalence of~\eqref{eq:kerperp} and~\eqref{eq:ker}. Clearly
	\begin{multline*}
	\Big\{\frac{\WSIP{\mathcal{P}Ax}{\mathcal{P}x}}{\seminorm{x}^2} \; | \; x \in \kerperp{\seminorm{\cdot}} \setminus \{\vectorzeros[n]\}\Big\} \\ \subseteq \Big\{\frac{\WSIP{\mathcal{P}Ax}{\mathcal{P}x}}{\seminorm{x}^2} \; | \; x \notin \kernel \seminorm{\cdot}\Big\}.
	\end{multline*} 
	Regarding the other containment, let $x \notin \kernel \seminorm{\cdot}$. Then since $\real^n = \kernel \seminorm{\cdot} \oplus \kerperp{\seminorm{\cdot}}$, there exist $x_{\kernel} \in \kernel \seminorm{\cdot}, x_{\perp} \in \kerperp{\seminorm{\cdot}} \setminus \{\vectorzeros[n]\}$ such that $x = x_{\kernel} + x_{\perp}$. Then
	\begin{align*}
	\frac{\WSIP{\mathcal{P}Ax}{\mathcal{P}x}}{\seminorm{x}^2} &= \frac{\WSIP{\mathcal{P}A(x_{\kernel} + x_\perp)}{\mathcal{P}(x_{\kernel}+ x_\perp)}}{\|\mathcal{P}(x_{\kernel} + x_\perp)\|^2} \\
	&= \frac{\WSIP{\mathcal{P}Ax_\perp}{\mathcal{P}x_\perp}}{\seminorm{x_\perp}^2}.
	\end{align*}
	Where the second equality holds because $\kernel \seminorm{\cdot} = \kernel \mathcal{P}$ and $A \kernel \seminorm{\cdot} \subseteq \kernel \seminorm{\cdot}$. Therefore,
	\begin{multline*}
	\Big\{\frac{\WSIP{\mathcal{P}Ax}{\mathcal{P}x}}{\seminorm{x}^2} \; | \; x \in \kerperp{\seminorm{\cdot}} \setminus \{\vectorzeros[n]\}\Big\} \\ \supseteq \Big\{\frac{\WSIP{\mathcal{P}Ax}{\mathcal{P}x}}{\seminorm{x}^2} \; | \; x \notin \kernel \seminorm{\cdot}\Big\}.
	\end{multline*} 
	This proves that the two sets are equal and therefore their supremums are equal. Next, we show that $\semimeasure{A} \geq \sup_{x \in \kerperp{\seminorm{\cdot}} \setminus \{\vectorzeros[n]\}} \frac{\WSIP{\mathcal{P}Ax}{\mathcal{P}x}}{\seminorm{x}^2}$. By Deimling's inequality,~\eqref{eq:DeimlingIneq}, for every $x \in \kerperp{\seminorm{\cdot}} \setminus \{\vectorzeros[n]\}$,
	\begin{align*}
	\WSIP{\mathcal{P}Ax}{\mathcal{P}x} &\leq \|\mathcal{P}x\|\lim_{h \to 0^+} \frac{\|\mathcal{P}x + h\mathcal{P}Ax\| - \|\mathcal{P}x\|}{h} \\
	&= \seminorm{x}\lim_{h \to 0^+} \frac{\seminorm{(I_n + hA)x} - \seminorm{x}}{h} \\
	&\leq \seminorm{x}^2 \lim_{h \to 0^+} \frac{\seminorm{I_n + hA}-1}{h} = \seminorm{x}^2\semimeasure{A},
	\end{align*}
	where the third line follows since $\seminorm{Ax} \leq \seminorm{A}\seminorm{x}$ for every $x \in \kerperp{\seminorm{\cdot}}$. This proves this inequality. Regarding the other direction, for $v \notin \kernel \seminorm{\cdot}$, define $\Omega(v) = \WSIP{\mathcal{P}Av}{\mathcal{P}v}/\seminorm{v}^2$.Then for every $v \notin \kernel \seminorm{\cdot}$,
	\begin{align*}
	\seminorm{(I_n - hA)v} &\geq \frac{1}{\seminorm{v}}\WSIP{\mathcal{P}(I_n-hA)v}{\mathcal{P}v} \\&\kern-4ex\geq (1 - h\Omega(v))\seminorm{v} \geq (1 - h\sup_{v \notin \kernel \seminorm{\cdot}}\Omega(v))\seminorm{v},
	\end{align*} 
	where the first inequality holds because of Cauchy-Schwarz, the second because of subadditivity of the weak pairing, and the final one because $-h < 0$. Then for small enough $h > 0$, $I_n - hA$ is invertible and given by $(I_n - hA)^{-1} = I_n + hA + h^2A^2(I_n - hA)^{-1}$, which implies
	\begin{equation*}
	\seminorm{(I_n + hA)v} \leq \seminorm{(I_n - hA)^{-1}v} + h^2\seminorm{A^2(I_n - hA)^{-1}v}, 
	\end{equation*}
	where the implication holds for all $v \in \real^n$ because of the triangle inequality for semi-norms. Moreover, let $x \in \kerperp{\seminorm{\cdot}} \setminus \{\vectorzeros[n]\}$ and define $v = (I_n - hA)^{-1}x$. To see that $v \notin \kernel \seminorm{\cdot}$, suppose for contradiction's sake that $v \in \kernel \seminorm{\cdot}$. Then since $(I_n - hA)v = x$, we have $v - hAv = x$. But $hAv \in \kernel \seminorm{\cdot}$ since, by assumption, $A \kernel \seminorm{\cdot} \subseteq \kernel \seminorm{\cdot}$. And since $\kernel \seminorm{\cdot}$ is a subspace, $v - hAv = x \in \kernel \seminorm{\cdot}$, a contraction. Therefore $v \notin \kernel \seminorm{\cdot}$. Then for $h > 0$, we have
	\begin{multline}\label{eq:semiinverse}
	\frac{\seminorm{(I_n - hA)^{-1}x}}{\seminorm{x}} = \frac{\seminorm{v}}{\seminorm{(I_n - hA)v}} \\\leq \frac{1}{1 - h\sup_{v \notin \kernel \seminorm{\cdot}}\Omega(v)}.
	\end{multline}
	Then
	\begin{align*}
	&\semimeasure{A} = \lim_{h \to 0^+} \sup_{x \in \kerperp{\seminorm{\cdot}} \setminus \{\vectorzeros[n]\}} \frac{\seminorm{(I_n + hA)x}/\seminorm{x} - 1}{h} \\
	&\leq \lim_{h \to 0^+} \sup_{x \in \kerperp{\seminorm{\cdot}} \setminus \{\vectorzeros[n]\}} \Big(\frac{\seminorm{(I_n - hA)^{-1}x}/\seminorm{x} - 1}{h} \\ & \qquad\qquad\qquad\qquad\qquad + \frac{h^2\seminorm{A^2(I_n - hA)^{-1}x}}{h\seminorm{x}}\Big) \\
	&\leq \lim_{h \to 0^+} \sup_{x \in \kerperp{\seminorm{\cdot}} \setminus \{\vectorzeros[n]\}} \frac{\seminorm{(I_n - hA)^{-1}x}/\seminorm{x} - 1}{h} \\
	&\leq \lim_{h \to 0^+} \frac{1}{h}\Big(\frac{1}{1 - h \sup\nolimits_{v \notin \kernel \seminorm{\cdot}}\Omega(v)} - 1\Big) \\&= \sup_{x \notin \kernel \seminorm{\cdot}} \frac{\WSIP{\mathcal{P}Ax}{\mathcal{P}x}}{\seminorm{x}^2},
	\end{align*}
	where the second line holds because of the triangle inequality, the third line holds because of the subadditivity of the supremum, and the fourth line holds because the inequality in~\eqref{eq:semiinverse} holds for all $x \in \kerperp{\seminorm{\cdot}} \setminus \{\vectorzeros[n]\}$. This proves the result.
\end{proof}

\begin{lemma}[Coppel's differential inequality for semi-norms]\label{lemma:Coppel}
	Let $\seminorm{\cdot} = \|\mathcal{P}\cdot\|$ be a seminorm on $\real^n$. Consider the dynamical system $\dot{x} = A(t,x)x$ where $(t,x) \mapsto A(t,x) \in \real^{n \times n}$ is continuous in $(t,x)$. Moreover, assume for all $x \in \real^n, t \in \realnonnegative$, $\kernel \seminorm{\cdot}$ is invariant under $A(t,x)$ in the sense that $A(t,x)\kernel \seminorm{\cdot} \subseteq \kernel\seminorm{\cdot}$. Then
	$$D^+\seminorm{x(t)} \leq \semimeasure{A(t,x(t))}\seminorm{x(t)}.$$
\end{lemma}
\begin{proof}
	First assume that $\seminorm{x(t)} = 0$ for some $t \geq 0$. Then $x(t) \in \kernel \seminorm{\cdot}$, and by assumption $A(t,x)x(t) \in \kernel \seminorm{\cdot}$. Then by computation:
	\begin{align*}
	D^+\seminorm{x(t)} &= \limsup_{h \to 0^+}\frac{\seminorm{x(t+h)} - \seminorm{x(t)}}{h} \\&= \limsup_{h \to 0^+}\frac{\seminorm{x(t) + hA(t,x(t))x(t)}}{h} \\
	&\leq \limsup_{h \to 0^+} \frac{|h|\seminorm{A(t,x)x(t)}}{h} = \seminorm{A(t,x)x(t)} = 0.
	\end{align*}
	So the result holds in this case. So assume $\seminorm{x(t)} \neq 0$, i.e., $x(t) \notin \kernel \seminorm{\cdot}$. Then apply the curve norm derivative formula for Deimling pairings, Lemma~\ref{lemma:DeimlingDerivative} to $\seminorm{x(t)}$:
	\begin{align*}
	\seminorm{x(t)}D^+\seminorm{x(t)} &= (\mathcal{P}A(t,x(t))x(t),\mathcal{P}x(t))_+  \\ \implies D^+\seminorm{x(t)} &= \frac{(\mathcal{P}A(t,x(t))x(t),\mathcal{P}x(t))_+}{\seminorm{x(t)}^2}\seminorm{x(t)} \\
	&\leq \semimeasure{A(t,x(t))}\seminorm{x(t)},
	\end{align*}
	where the inequality holds by Lemma~\ref{lemma:semiLumer} since $(\cdot,\cdot)_+$ is a weak pairing that trivially satisfies Deimling's inequality,~\eqref{eq:DeimlingIneq}.
\end{proof}

We now prove semi-contraction theorems analogous to the contraction theorem for continuously differentiable vector fields, Theorem~\ref{thm:general}, and the equilibrium contraction theorem, Theorem~\ref{aveContr}.

\begin{theorem}[Semi-contraction theorem for continuously differentiable vector fields]
	\label{thm:semi}
	Consider the dynamical system $\dot{x} = f(t,x),$ which is continuously differentiable in $x$ with Jacobian $\jac{f}$. Let $\seminorm{\cdot} = \|\mathcal{P}\cdot\|$ be a semi-norm on $\real^n$. Let $\WSIP{\cdot}{\cdot}$ be a weak pairing on $\kerperp{\seminorm{\cdot}}$ satisfying Deimling's inequality,~\eqref{eq:DeimlingIneq}. Moreover, assume $\jac{f}(t,x)\kernel\seminorm{\cdot} \subseteq \kernel \seminorm{\cdot}$ for all $t \geq 0, x \in \real^n$. Then, for $b \in \real$, the following statements are equivalent:
	\begin{enumerate}
		\item\label{semi:0}
		$\semimeasure{\jac{f}(t,x)} \leq b$, for all $x \in \real^n, t \geq 0$,
		\item\label{semi:1} $\WSIP{\mathcal{P}\jac{f}(t,x)v}{\mathcal{P}v} \leq b\seminorm{v}^2$, for all $v \in \R^n, x \in \real^n, t \geq 0$, 
		\item\label{semi:2} $\WSIP{\mathcal{P}(f(t,x) - f(t,y))}{\mathcal{P}(x - y)} \leq b\seminorm{x - y}^2$, for all $x,y \in \real^n, t \geq 0,$
		\item\label{semi:3} $D^+\vert\kern-0.25ex\vert\kern-0.25ex\vert\phi(t,t_0,x_0) - \phi(t,t_0,y_0)\vert\kern-0.25ex\vert\kern-0.25ex\vert \leq b\vert\kern-0.25ex\vert\kern-0.25ex\vert\phi(t,t_0,x_0) - \phi(t,t_0,y_0)\vert\kern-0.25ex\vert\kern-0.25ex\vert,$ for all $x_0,y_0 \in \real^n$, $0 \leq t_0 \leq t$ for which solutions exist,
		\item\label{semi:4} $\vert\kern-0.25ex\vert\kern-0.25ex\vert\phi(t,t_0,x_0) - \phi(t,t_0,y_0)\vert\kern-0.25ex\vert\kern-0.25ex\vert \leq e^{b(t-s)}\vert\kern-0.25ex\vert\kern-0.25ex\vert\phi(s,t_0,x_0) - \phi(s,t_0,y_0)\vert\kern-0.25ex\vert\kern-0.25ex\vert$, for all $x_0,y_0 \in \real^n$, $0 \leq t_0 \leq s \leq t$ for which solutions exist.
	\end{enumerate}
\end{theorem}
\begin{proof}
	Regarding \ref{semi:0} $\iff$ \ref{semi:1}, the proof follows by Lumer's equality for log semi-norms, Lemma~\ref{lemma:semiLumer}.
	
	Regarding \ref{semi:1} $\implies$ \ref{semi:2}, suppose $\WSIP{\mathcal{P}\jac{f}(t,z)v}{\mathcal{P}v} \leq b\seminorm{v}^2$, for all $v \in \real^n, z \in \real^n, t \geq 0$. Then let $x, y \in \real^n$ and $v = x - y$. By the mean-value theorem for vector-valued functions, we have
	$$f(t,x) - f(t,y) = \left(\int_0^1 \jac{f}(t,y + sv)ds\right)(x - y).$$
	Hence,
	\begin{align*}
	&\WSIP{\mathcal{P}(f(t,x) - f(t,y))}{\mathcal{P}(x-y)} \\&= \WSIP{\mathcal{P}\left(\int_0^1 \jac{f}(t,y + sv)ds\right)(x - y)}{\mathcal{P}(x - y)} \\
	&= \WSIP{\left(\int_0^1 \mathcal{P}\jac{f}(t,y + sv)ds\right)(x - y)}{\mathcal{P}(x - y)} \\
	& \leq \int_{0}^1 \WSIP{\mathcal{P}\jac{f}(t,y + sv)(x - y)}{\mathcal{P}(x -y)}ds \\
	& \leq \int_0^1 b\seminorm{x - y}^2ds = b\seminorm{x-y}^2,
	\end{align*}
	where the third line follows from the sublinearity and continuity in the first argument of the weak pairing. 
	
	Regarding \ref{semi:2} $\implies$ \ref{semi:1}, assume $\WSIP{\mathcal{P}(f(t,x) - f(t,y))}{\mathcal{P}(x-y)} \leq b\seminorm{x-y}^2$, for all $x,y \in \real^n, t \geq 0$. If $x = y$, the result is trivial, so assume $x \neq y$. Fix $y$ and set $x = y + hv$ for an arbitrary $v \in \R^n, h \in \realpositive$. Then substituting
	\begin{align*}
	\WSIP{\mathcal{P}(f(t,y + hv) - f(t,y))}{\mathcal{P}hv} &\leq b\seminorm{hv}^2 \\\implies \quad h\WSIP{\mathcal{P}(f(t,y + hv) - f(t,y))}{\mathcal{P}v} &\leq bh^2\seminorm{v}^2,
	\end{align*}
	by the weak homogeneity of the weak pairing. Dividing by $h^2$ and taking the limit as $h$ goes to zero yields
	\begin{align*}
	\lim_{h \to 0^+} \WSIP{\mathcal{P}\frac{f(t,y + hv) - f(t,y)}{h}}{\mathcal{P}v} &\leq b\seminorm{v}^2  \\\implies \quad \WSIP{\mathcal{P}\jac{f}(t,y)v}{\mathcal{P}v} &\leq b\seminorm{v}^2,
	\end{align*}
	which follows from the continuity of the weak pairing in its first argument. Since $y$, $v$, and $t$ were arbitrary, this completes this implication. 
	
	Regarding \ref{semi:0} $\implies$ \ref{semi:3}, 
	let $x_0,y_0 \in \real^n, t_0 \geq 0$ and let $x(t) = \phi(t,t_0,x_0), y(t) = \phi(t,t_0,y_0)$. Then by the mean-value theorem for vector-valued functions, for $v(t) = x(t) - y(t)$,
	$$\dot{v} = \Big(\int_0^1 \jac{f}(t,y + sv)ds\Big)v.$$
	Then since $\jac{f}(t,x)\kernel \seminorm{\cdot} \subseteq \kernel \seminorm{\cdot}$ for all $t \geq 0, x \in \real^n$, by Coppel's differential inequality for semi-norms, Lemma \ref{lemma:Coppel}, 
	\begin{align*}
	D^+ \seminorm{v(t)} &\leq \semimeasure{\int_0^1 \jac{f}(t,y(t) + sv(t))ds} \seminorm{v(t)} \\
	&\leq \int_0^1 \semimeasure{\jac{f}(t,y(t) + sv(t))} ds \;\seminorm{v(t)} \\
	&\leq \int_0^1 b ds \; \seminorm{v(t)} = b\seminorm{v(t)},
	\end{align*}
	which holds by the subadditivity of log semi-norms.
	
	Regarding \ref{semi:3} $\implies$ \ref{semi:4}, the result follows by the nonsmooth \GB inequality. 
	
	Regarding \ref{semi:4} $\implies$ \ref{semi:2}, let $x_0,y_0 \in C$, $t_0 \geq 0$ be arbitrary. Then for $h \geq 0$,
	\begin{align*}
	&\seminorm{\phi(t_0 + h,t_0,x_0) - \phi(t_0 + h,t_0,y_0)} \\&\qquad = \seminorm{x_0 - y_0 + h(f(t_0,x_0) - f(t_0,y_0))} + O(h^2) \\
	&\qquad \leq e^{bh}\seminorm{x_0 - y_0}.
	\end{align*}
	Subtracting $\seminorm{x_0 - y_0}$ on both sides, dividing by $h > 0$ and taking the limit as $h \to 0^+$, we get
	\begin{align*}
	&\lim_{h \to 0^+} \frac{\seminorm{x_0 - y_0 + h(f(t_0,x_0) - f(t_0,y_0))} - \seminorm{x_0 - y_0}}{h} \\ &\qquad \leq \lim_{h \to 0^+} \frac{e^{bh} - 1}{h}\seminorm{x_0 - y_0}.
	\end{align*}
	Evaluating the right hand side limit gives
	\begin{align*}
	&\lim_{h \to 0^+} \frac{\seminorm{x_0 - y_0 + h(f(t_0,x_0) - f(t_0,y_0))} - \seminorm{x_0 - y_0}}{h} \nonumber \\&\qquad \leq b\seminorm{x_0 - y_0}.
	\end{align*}
	But by the assumption of Deimling's inequality,~\eqref{eq:DeimlingIneq}, multiplying both sides by $\seminorm{x_0 - y_0}$ gives
	$$\WSIP{\mathcal{P}(f(t_0,x_0) - f(t_0,y_0))}{\mathcal{P}(x_0-y_0)} \leq b\seminorm{x_0 - y_0}^2.$$
	Since $t_0, x_0$ and $y_0$ were arbitrary, the result holds.
\end{proof}

\begin{remark}
	In~\cite[Theorem~13]{SJ-PCV-FB:19q},~\ref{semi:0} $\implies$~\ref{semi:4} is proved. Theorem~\ref{thm:semi} generalizes this result and provides necessary and sufficient conditions for semi-contraction analogous to those in Theorem~\ref{thm:general} including a one-sided Lipschitz condition on the vector field via the semi-norm.
\end{remark}

\begin{theorem}[Subspace contraction equivalences]
	\label{thm:subspace}
	Consider the dynamical system $\dot{x} = f(t,x),$ which is continuous in $(t,x)$. Let $\seminorm{\cdot} = \|\mathcal{P}\cdot\|$ be a semi-norm on $\real^n$. Let $\WSIP{\cdot}{\cdot}$ be a weak pairing on $\kerperp{\seminorm{\cdot}}$ satisfying Deimling's inequality,~\eqref{eq:DeimlingIneq}, and the curve norm derivative formula,~\eqref{eq:curvenormderivative}. Moreover, assume that there exists $x^* \in \kerperp{\seminorm{\cdot}}$ such that $f(t,x^* + \kernel \seminorm{\cdot}) \subseteq \kernel \seminorm{\cdot}$ for all $t \geq 0$. Then, for $b \in \real$, the following statements are equivalent:
	\begin{enumerate} 
		\item\label{subspace:1} $\WSIP{\mathcal{P}f(t,x)}{\mathcal{P}(x - x^*)} \leq b\seminorm{x - x^*}^2$, for all $x \in \real^n, t \geq 0,$
		\item\label{subspace:2} $D^+\seminorm{\phi(t,t_0,x_0) - x^*} \leq b\seminorm{\phi(t,t_0,x_0) - x^*},$ for all $x_0$, $0 \leq t_0 \leq t$ for which solutions exist,
		\item\label{subspace:3} $\seminorm{\phi(t,t_0,x_0) - x^*} \leq e^{b(t-s)}\seminorm{\phi(s,t_0,x_0) - x^*}$, for all $x_0 \in \real^n$, $0 \leq t_0 \leq s \leq t$ for which solutions exist.
	\end{enumerate}
	Moreover, suppose that $f(t,x)$ is continuously differentiable in $x$ and that $\semimeasure{\jac{f}(t,x)} \leq b$ for all $(t,x)$. Then statements~\ref{subspace:1},~\ref{subspace:2}, and~\ref{subspace:3} hold.
\end{theorem}
\begin{proof}
	Regarding \ref{subspace:1} $\implies$ \ref{subspace:3}, 
	let $x_0 \in \real^n, t_0 \geq 0$ and let $x(t) = \phi(t,t_0,x_0), y(t) = \phi(t,t_0,y_0)$. Then we have two cases. If there exists $t \geq t_0$ such that $\seminorm{x(t) - x^*} = 0$, then note that $x(t) - x^* \in \kernel \seminorm{\cdot}$, which implies $x(t) \in x^* + \kernel \seminorm{\cdot}$, which further implies $\seminorm{f(t,x(t))} = 0$ by assumption. Then we compute
	\begin{align*}
	D^+\seminorm{x(t) - x^*} &= \limsup_{h \to 0^+} \frac{\seminorm{x(t+h) - x^*} - \seminorm{x(t) - x^*}}{h} \\
	&\!\!\!\!\!\!= \limsup_{h \to 0^+} \frac{\seminorm{x(t) - x^* + hf(t,x(t))}}{h} \\
	&\!\!\!\!\!\!\leq \limsup_{h \to 0^+} \frac{|h|\seminorm{f(t,x(t))}}{h} = \seminorm{f(t,x(t))} = 0.
	\end{align*}
	Then for all $t \geq t_0$ for which $\seminorm{x(t) - x^*} \neq 0$, apply the curve norm derivative to get
	\begin{align*}
	\seminorm{x(t) - x^*}D^+\seminorm{x(t) - x^*} &= \WSIP{\mathcal{P}f(t,x(t))}{\mathcal{P}(x(t) - x^*)} \\
	\implies \quad D^+\seminorm{x(t) - x^*} &\leq b\seminorm{x(t) - x^*},
	\end{align*}
	for almost every $t \geq 0$. Then applying the nonsmooth \GB inequality, Lemma~\ref{lemma:nonsmooth-GB} gives the desired result.
		
	Regarding \ref{subspace:2} $\implies$ \ref{subspace:3}, the result follows by the nonsmooth \GB inequality. 
	
	Regarding \ref{subspace:3} $\implies$ \ref{subspace:1}, let $x_0 \in C$, $t_0 \geq 0$ be arbitrary. Then for $h \geq 0$,
	\begin{align*}
	&\seminorm{\phi(t_0 + h,t_0,x_0) - x^*} \\&\qquad = \seminorm{x_0 - x^* + hf(t_0,x_0)} + O(h^2) \\
	&\qquad \leq e^{bh}\seminorm{x_0 - x^*}.
	\end{align*}
	Subtracting $\seminorm{x_0 - x^*}$ on both sides, dividing by $h > 0$ and taking the limit as $h \to 0^+$, we get
	\begin{align*}
	&\lim_{h \to 0^+} \frac{\seminorm{x_0 - x^* + hf(t_0,x_0)} - \seminorm{x_0 - x^*}}{h} \\ &\qquad \leq \lim_{h \to 0^+} \frac{e^{bh} - 1}{h}\seminorm{x_0 - x^*}.
	\end{align*}
	Evaluating the right hand side limit gives
	\begin{align*}
	&\lim_{h \to 0^+} \frac{\seminorm{x_0 - x^* + hf(t_0,x_0)} - \seminorm{x_0 - x^*}}{h} \nonumber \\&\qquad \leq b\seminorm{x_0 - x^*}.
	\end{align*}
	But by the assumption of Deimling's inequality,~\eqref{eq:DeimlingIneq}, multiplying both sides by $\seminorm{x_0 - x^*}$ gives
	$$\WSIP{\mathcal{P}f(t_0,x_0)}{\mathcal{P}(x_0-x^*)} \leq b\seminorm{x_0 - x^*}^2.$$
	Since $t_0$ and $x_0$ were arbitrary, the result holds.
	
	Regarding \ref{subspace:3} $\implies$ \ref{subspace:2}, following the proof of \ref{subspace:3} $\implies$ \ref{subspace:1}, the inequality
	\begin{equation}
	(\mathcal{P}f(t,x), \mathcal{P}(x - x^*))_+ \leq b\seminorm{x - x^*}^2,
	\end{equation}
	holds for all $x \in \real^n, t \geq 0$, where the Deimling pairing is defined on $\kerperp{\seminorm{\cdot}}$. Then applying the curve norm derivative for the Deimling pairing, Lemma~\ref{lemma:DeimlingDerivative} on $\seminorm{\phi(t,t_0,x_0) - x^*}$ proves the result.
	
	Further, suppose $f(t,x)$ is continuously differentiable and $\semimeasure{\jac{f}(t,x)} \leq b$ for all $x \in \real^n, t \geq 0$. Then~\cite[Theorem~13(ii)]{SJ-PCV-FB:19q} implies that~\ref{subspace:3} holds.
\end{proof}
\begin{remark}
	Subspace contraction is a form of equilibrium contraction for semi-contracting systems. In~\cite[Theorem~13]{SJ-PCV-FB:19q}, $\semimeasure{\jac{f}(t,x)} \leq b \implies$~\ref{subspace:3} is proved. However, Theorem~\ref{thm:subspace} gives necessary and sufficient conditions analogous to those in Theorem~\ref{aveContr} and demonstrates that weak pairings are a suitable tool for checking semi-contraction and subspace contraction.
\end{remark}
\end{arxiv}

\bibliography{alias,Main,FB}

\begin{tac}
\begin{IEEEbiography}
	[{\includegraphics[width=1in,height=1.25in,clip,keepaspectratio]{alexander-davydov}}]{Alexander Davydov}
	(S'21) is a a PhD student at the University of California, Santa Barbara with the Mechanical Engineering Department and the Center for Control, Dynamical Systems, and Computation. He received B.S. degrees in Mechanical Engineering and Mathematics from the University of Maryland. His research interests include the analysis and control of nonlinear network systems with applications to machine learning, neural networks, and multi-robot systems.
\end{IEEEbiography}

\begin{IEEEbiography}
	[{\includegraphics[width=1in,height=1.25in,clip,keepaspectratio]{saber-jafarpour}}]{Saber Jafarpour}
	(M’16) is a Postdoctoral researcher with the Mechanical Engineering Department and the Center for Control, Dynamical Systems
	and Computation at the University of California,
	Santa Barbara. He received his Ph.D. in 2016 from
	the Department of Mathematics and Statistics at
	Queen’s University. His research interests include
	analysis of network systems with application to
	power grids and geometric control theory.
\end{IEEEbiography}

\begin{IEEEbiography}
	[{\includegraphics[width=1in,height=1.25in,clip,keepaspectratio]{francesco-bullo}}]{Francesco Bullo}
	(S’95–M’99–SM’03–F’10) is a Professor with the Mechanical Engineering Department and the Center for Control, Dynamical Systems and Computation at the University of California, Santa Barbara. He was previously associated with the University of Padova (Laurea degree in Electrical Engineering, 1994), the California Institute of Technology (Ph.D. degree in Control and Dynamical Systems, 1999), and the University of Illinois. He served on the editorial boards of IEEE, SIAM, and ESAIM journals and as IEEE CSS President. His research interests focus on network systems and distributed control with application to robotic coordination, power grids and social networks. He is the coauthor of “Geometric Control of Mechanical Systems” (Springer, 2004), “Distributed Control of Robotic Networks” (Princeton, 2009), and “Lectures on Network Systems” (Kindle Direct Publishing, 2022, v1.6). He received best paper awards for his work in IEEE Control Systems, Automatica, SIAM Journal on Control and Optimization, IEEE Transactions on Circuits and Systems, and IEEE Transactions on Control of Network Systems. He is a Fellow of IEEE, IFAC, and SIAM.
\end{IEEEbiography}
\end{tac}
      
\begin{extensions}
	\clearpage
	
	\subsection{Future directions}
	The following is a list of possible future directions, theoretical
	extensions, and applications of this work.
	\begin{enumerate}
		\item DONE: Monotone case, i.e., either $\jac{f}$ is Metzler or $f$
		satisfies the Kamke-M\"uller conditions. This extension has been
		partially explored by Saber.
		
		\item PARTLY DONE??? Semi-contraction. Can the theory of WSIPs extend to
		weight matrices that are rank deficient? If the semi-norm defines a norm
		on a lower-dimensional vector space, Lumer's lemma should apply on that
		subspace. Then maybe we can lift the SIP/WSIP to the higher space to do
		computations.
		
		\begin{itemize}
			\item Semi-contraction of neural ODEs to realistic, known subspaces. For
			instance, take the problem of trying to learn flows in an
			infrastructure network. Then we know that some commodity must be
			preserved, so solutions must live in $\vectorones[n]^\perp$. Thus, have
			the neural ODE be semi-contracting with respect to that subspace.
			
			\item Cluster synchronization. Either related to Mario Di Bernardo's work
			or recently Jeff Moehlis's work on clustering neurons to a splay state.
		\end{itemize}
		
		\item Contraction on Riemannian and Finsler manifolds. How does the theory
		of WSIPs extend to state-dependent weight matrices? Can we come up with a
		good example for which this machinery is necessary? Possibly some sort of
		Kuramoto model on the $n$-torus?
		
		\item Contraction on Banach spaces. Lumer's lemma holds for complete normed
		spaces, i.e., Banach spaces. Pedro had some results for contraction on
		Hilbert spaces, but it seems like many results from SIP and WSIP theory
		should extend in an analogous way to infinite-dimensional linear spaces.
		
		\item Computing equilibrium points for contracting systems. For globally
		Lipschitz vector fields which are contracting with respect to the
		$\ell_2$ norm, it is well-known that a constant stepsize forward Euler
		method converges to the equilibrium exponentially quickly and this
		optimal stepsize may be computed. Equivalent results do not seem to be
		known for vector fields which are contracting with respect to
		non-Euclidean norms.
		
		\item Contraction for discontinuous dynamical systems and differential
		inclusions. We can still define the one-sided Lipschitz condition for
		contraction for discontinuous dynamical systems and differential
		inclusions. Does this guarantee incremental stability between Filippov
		solutions? Is the condition still necessary and sufficient? Likely it is
		sufficient, but may not be necessary. A Filippov solution need only be
		absolutely continuous, so the proof given in Theorem \ref{thm:general}
		does not apply since we cannot Taylor expand $x(t+h)$ and $y(t+h)$.
		\begin{itemize}
			\item Can we study negative gradient and subgradient flows for nonsmooth
			potential functions such as $V(x) = \|x\|_1$ or $\|x\|_{\infty}$? Are
			their negative gradient flows at least weakly contracting?
			\item Optimization on manifolds? I have not yet fully understood this
			application, but this could certainly be a natural extension of any
			optimization results we find.

                      \item May 2021: Filippov, page 5. Can we prove
                        uniquness for discontinuous vector fields?
                        
		\end{itemize}
	      \item Contraction for regression and regularization.
              \item is it worth defining $\operatorname{osL}(f)$?
	\end{enumerate}
	
	\newpage
	
	Suppose that $x=(x_1,\ldots,x_n)$ is a random variable in $\real^n$
	such that, for every $i\in \{1,\ldots,n\}$, we have $x_i \in
	\mathcal{N}(0,\sigma)$. Then we have
	\begin{align*}
	\mathbb{E}(\|x\|^2_2) &= \sum_{i=1}^{n}\mathbb{E}(x^2_i) =
	\sum_{i=1}^{n}\sigma = n\sigma.\\
	\mathbb{E}(\|x\|^2_{\infty}) &= \mathbb{E}(\max_{i}|x_i|^2) =
	\max_{i}  \mathbb{E}(|x_i|^2) =\sigma
	\end{align*}

\end{extensions}

\end{document}